\documentclass[11pt]{article}
\usepackage{macros}

\theoremstyle{plain}
\newtheorem{theorem}{Theorem}[section]
\newtheorem{proposition}[theorem]{Proposition}

\newtheorem{lemma}[theorem]{Lemma}
\newtheorem{claim}[theorem]{Claim}

\theoremstyle{definition}
\newtheorem{definition}[theorem]{Definition}

\theoremstyle{remark}
\newtheorem{remark}[theorem]{Remark}

\definecolor{AfonsoBlue}{RGB}{30,65,123}
\definecolor{TimGreen}{rgb}{0.172549,0.627451,0.172549}

\definecolor{PNNOrange}{RGB}{242,140,40}

\title{Convergence of Kikuchi matrices\\[-0.15em]
  to $\Gamma$-independent and $q$-Gaussian limits}

\author{%
  \makebox[\textwidth][c]{%
    Afonso S. Bandeira\thanks{Department of Mathematics, ETH Z\"urich. \texttt{bandeira@math.ethz.ch}}
    \qquad
    Dmitriy Kunisky\thanks{Department of Applied Mathematics \& Statistics, Johns Hopkins University. \texttt{kunisky@jhu.edu}}
    \qquad
    Petar Nizi\'{c}-Nikolac\thanks{Department of Mathematics, ETH Z\"urich. \texttt{petar.nizic-nikolac@ifor.math.ethz.ch}}
  }\\[0.8em]
  \makebox[\textwidth][c]{%
    Lucas Pesenti\thanks{Department of Mathematics, ETH Z\"urich. \texttt{lpesenti@ethz.ch}}
    \qquad\qquad
    Robert Wang\thanks{Cheriton School of Computer Science, University of Waterloo. \texttt{robert.wang2@uwaterloo.ca}}
  }
}

\date{\today}

\begin{document}

\maketitle

\begin{abstract}
    Kikuchi matrices are a family of structured matrices that were introduced to study problems involving tensors and hypergraphs.
    We show that, as the ambient dimension
    grows, dense random Kikuchi matrices
    have a limit described by a system of $\Gamma$-independent semicircular elements.
    This characterizes their limiting spectral distribution and yields improved bounds on their spectral norm, a key quantity in the analysis of algorithms for Tensor PCA.
    Finally, we show that, in an appropriate double limit, independent Kikuchi matrices converge to the $q$-Gaussian system, another central object in noncommutative probability.
\end{abstract}

\section{Introduction}

Kikuchi matrices are a family of structured matrices that were recently
introduced to reduce problems involving higher-degree polynomials, hypergraphs, or tensors to more tractable quadratic, graph, or matrix problems.
For many applications, the bottleneck in this approach lies in precisely estimating the spectral norm
of Kikuchi matrices as the ambient dimension
grows. 
In this paper, we study the
{\em limit} of Kikuchi matrices, focusing on the dense random Kikuchi matrices that arise in the analysis of algorithms for tensor PCA~\cite{WeinElAlaouiMoore2019,tensorPCAQuantum,BCSvH24,Manohar-2024-KikuchiMatrixMethod,kothari2025smoothtradeofftensorpca}. 

Given integers $n,\ell,r$ satisfying
$1\le r\le \ell\le n-r$, the 
{\em Kikuchi matrix} $Y_n=Y_n(\ell,r)$ 
is the random matrix with entries
\begin{equation}\label{eq:KikuchiYSTconstruction}
    Y_n[U,V] \defeq \begin{cases}  \dfrac 1 {\sqrt{d_n}} \;g_{U\mathbin\triangle V} & \text{if $|U\mathbin \triangle V|=2r$}\\
    0 & \text{otherwise}
    \end{cases}\qquad \text{for }U,V\in \binom {[n]} \ell\,,
\end{equation}
where $g_I$ are independent standard Gaussian random variables for all $I\subseteq [n]$ of size $|I|=2r$. Here, $U\mathbin \triangle V = (U\cup V)\setminus (U\cap V)$ denotes
the symmetric difference between sets, and
\begin{equation}
    d_n\defeq d_n(\ell,r)= \binom \ell r \binom {n-\ell} r\label{eq:degreeKikuchi}
\end{equation}
is the normalization that makes the expected squared $\ell_2$-norm of each row of the Kikuchi matrix equal to $1$.
Our goal is to describe the limit as $n\to\infty$ of the
spectrum and norm of $Y_n$ for fixed $\ell,r$. 

Henceforth and unless stated otherwise, asymptotics are with respect to $n\to\infty$ for fixed $\ell,r$. Similarly,
``with high probability'' should be understood as with probability tending to $1$ as $n\to\infty$.

\paragraph{Tensor PCA.} One motivation for understanding the spectrum of the Kikuchi
matrix is the {\em tensor PCA} problem~\cite{tensorPCADefinition}. In one of its idealized forms, we aim to recover a vector of spins $x\in\{-1,1\}^{n}$ from order-$2r$ tensor measurements. For each set $I\subseteq [n]$ of size $|I|=2r$, we observe
\[
    \mathcal{T}_I = \lambda_n \prod_{i\in I}x_i + g_I\,,
\]
where the $g_I$ are i.i.d.\ standard Gaussians and $\lambda_n> 0$ is the signal strength. To apply the ``Kikuchi matrix method'', given these
observations, for any $r\le \ell\le n-r$, one can form the matrix indexed by $U,V\in \binom {[n]} \ell$ with entries
\[
    M_n[U,V] \defeq \begin{cases}
        \mathcal{T}_{U\mathbin \triangle V} & \text{if $|U\mathbin \triangle V| = 2r$,}\\
        0 & \text{otherwise.}
    \end{cases}
\]
When $\lambda_n$ is large enough, $M_n$ has an outlier
eigenvector that can be used to recover $x$ up to a global sign.
To understand the spectrum of $M_n$, we may assume without loss of generality that $x$ is the all-ones
vector. Then $M_n = \lambda _nA_n + \sqrt {d_n} \,Y_n$, where $A_n$ is the
adjacency matrix of a regular graph (the Johnson graph) whose degree is $d_n$ defined in~\eqref{eq:degreeKikuchi}, and $Y_n$ is the Kikuchi matrix defined in~\eqref{eq:KikuchiYSTconstruction}. We
therefore
expect the signal to be visible in the spectrum
of $M_n$ whenever
$\lambda_n d_n > \sqrt d_n\, \|Y_n\|$; the analysis reduces to estimating the norm of $Y_n$.\footnote{Here we allow for a constant factor of inaccuracy in understanding the size of $\lambda_n$ necessary to create an outlier; understanding the exact threshold would raise the more complicated question of Baik-Ben Arous-P\'{e}ch\'{e} (BBP) transitions with a Kikuchi noise matrix, which we do not consider here.}
We point the reader to~\cite{BCSvH24} for a sharp
derivation of the detection threshold of tensor PCA
using this construction when $\ell<\frac 32r$.\footnote{The asymptotics of~\eqref{eq:degreeKikuchi} suggested an intriguing phenomenon for Tensor PCA: the detection threshold decreases as $\ell\to\infty$. This was the subject of~\cite[Conjecture 9]{bandeira2025randomstrasse101}, subsequently resolved in~\cite{kothari2025smoothtradeofftensorpca}. As a direct consequence of Theorem~\ref{thm:kothariXuImprovement} below, we further improve the dependence on $r$ to at most a polynomial factor away from optimal.}

\paragraph{Prior work.} 
We first review previous bounds on the norm of Kikuchi matrices.
A simple application of the
classical
noncommutative Khintchine inequality~\cite{NCK} gives, with high probability,
\begin{equation}\label{eq:usingNCK}
    \|Y_n\|\lesssim \sqrt{\log \binom n \ell} \lesssim \sqrt {\ell \log n}\,.
\end{equation}
This bound is used, for instance, in the first application of Kikuchi matrices to tensor PCA~\cite{WeinElAlaouiMoore2019}.
However, in many applications of this inequality, the logarithmic factor is known to be an artifact of the analysis and can be removed with more careful or application-specific calculations. 
To address this issue, a series of works~\cite{BBvH23,BCSvH24,bandeira2025randommatrices} developed improved noncommutative Khintchine inequalities based on ideas from free probability. In particular, when specialized to Kikuchi matrices, these inequalities 
show that, in the regime $r\le \ell< 2r$,
the spectrum of $Y_n$ is asymptotically semicircular and $\|Y_n\|=2+o(1)$ with high probability~\cite{BCSvH24}.
However, these techniques break down once $\ell\ge 2r$, a regime in which the Kikuchi matrix
does not asymptotically 
have a semicircular distribution of eigenvalues.
This was shown formally by Kothari and Xu~\cite{kothari2025smoothtradeofftensorpca}: using the trace method and a delicate 
charging argument, they
established that instead we have $\|Y_n\| \asymp_r \sqrt \ell$ in a large regime of $\ell\le n^{\delta(r)}$, up to constants depending only on $r$.

In this paper, we initiate a program to determine the {\em exact} asymptotic spectrum
and spectral norm of Kikuchi matrices, for arbitrary fixed $\ell$ and $r$. 

\paragraph{Weak and strong convergence.}
The limiting object associated
with a single Kikuchi matrix
is a bounded linear operator $A:\calH\to \calH$
on an infinite-dimensional Hilbert space $\calH$.
We refer the reader to Section~\ref{sec:operatorNotation} for our notation for operators on Hilbert spaces.
Our results are formulated more generally for tuples of
independent Kikuchi matrices, in terms of {\em weak} and {\em strong} convergence.

\begin{definition}
    Fix $k\ge 1$. Let $(\calH_n)_{n\ge 1}$ and $\calH$ be
    Hilbert spaces, with states $\tau_n$ on $\calB(\calH_n)$ and $\tau$ on $\calB(\calH)$.
    For each $n\ge 1$, let $\bd A_n = (A_{n,1},\ldots,A_{n,k})$ be a
    tuple of random self-adjoint elements of
    $\calB(\calH_n)$, and let $\bd A = (A_1, \ldots, A_k)$ be a tuple of self-adjoint elements of $\calB(\calH)$. We say that
    \begin{enumerate}
        \item $\bd A_n$ converges to $\bd A$ {\em weakly} if
        \[
            \tau_n(P(\bd A_n)) \underset{n\to\infty}\longrightarrow \tau(P(\bd A))\quad \text{in probability},
        \]
        for every noncommutative polynomial $P\in \C\iprod{X_1, \ldots, X_k}$.
        \item $\bd A_n$ converges to $\bd A$ {\em strongly} if $\bd A_n$ converges to $\bd A$ weakly, and
        \[
            \|P(\bd A_n)\| \underset{n\to\infty}\longrightarrow \|P(\bd A)\|\quad \text{in probability},
        \]
        for every noncommutative polynomial $P\in \C\iprod{X_1, \ldots, X_k}$.
    \end{enumerate}
\end{definition}

\noindent In the special case where $\calH_n$ is finite-dimensional, the random operators $A_{n,i}$ are random matrices, and we take $\tau_n=\tr$ to be the normalized trace (the usual trace divided by the dimension of the matrix), while $\norm{\,\cdot\,}$ is the usual spectral norm. For a single matrix, weak convergence corresponds to convergence of the spectral distribution, while strong convergence additionally captures convergence of the spectral norm.
Powerful new techniques for establishing
strong convergence have emerged in recent 
years~\cite{HT05,CGVTVH25,CGVvH24,vanhandel2025strongconvergence}.

\paragraph{$\Gamma$-independent systems.} Next, we define the candidate limiting operator.
It is built from a family of operators with semicircular distributions whose joint distribution satisfies a notion of independence known as \emph{$\Gamma$-independence}. This notion generalizes both free and classical independence.
The construction below is due to Bo\.zejko and Speicher~\cite{proofGammaRepresentation}, with a simplified formulation due to Oliveira Santos, Tripathi, and Youssef~\cite{youssefTuran}.

\begin{definition}\label{def:traceMonoid}
    Let $\Gamma=(V,E)$ be an undirected graph. The set of
    words $\bigsqcup_{k\ge 0} V^k$ on the alphabet $V$
    has a monoid structure given by
    concatenation. We denote by $w\cdot w'$
    the concatenation of two words $w,w'$.
    The {\em trace monoid} $\calT=\calT(\Gamma)$ is the
    quotient monoid obtained
    by identifying two words whenever one can be obtained from the other by repeatedly interchanging consecutive letters that are adjacent in $\Gamma$. We denote by $\varnothing$ the element of the
    trace monoid corresponding to the empty word.
\end{definition}

For example, if $\Gamma$ is the graph on $V=\{a,b,c\}$ with a single edge between $a$ and $b$, then we have $a\cdot b\cdot c = b\cdot a\cdot c$ in the trace monoid, whereas $a\cdot b\cdot c \neq a\cdot c\cdot b$. At the two extremes, the trace monoid is the free monoid when $\Gamma$ has no edges, and the free commutative monoid when $\Gamma$ is complete.

\begin{definition}\label{def:gammaSystem}
    Let $\calT$ be the trace monoid
    associated with the graph $\Gamma=(V,E)$.
    Consider the
    complex Hilbert space $\ell^2(\calT)$ with canonical basis $\{\delta_w : w\in \calT\}$. We equip $\calB(\ell^2(\calT))$ with the state
    $\tau(L) \defeq \iprod{\delta_\varnothing, L \delta_{\varnothing}}$.
    For each $u\in V$, let $L_u\in \calB(\ell^2(\calT))$ be the {\em creation operator}, which is the linear map such that $L_u(\delta_w) = \delta_{w\cdot u}$. We call $(L_u + L_u^\ast)_{u\in V}$ a {\em $\Gamma$-independent semicircular system}. 
\end{definition}

Here, $L_u^\ast$ denotes the adjoint of
the creation operator $L_u$. On a basis element,
the effect of $L_u^\ast$ is to remove an occurrence of $u$ on the right of the word, if possible, and to map to zero otherwise:
\[
    L_u^\ast (\delta_w) = \begin{cases}\delta_{w'} & \text{ if $w = w'\cdot u$ for some $w'\in \calT$,}\\
    0 & \text{ otherwise.} \end{cases}
\]
The factorization $w=w'\cdot u$ is understood as an equality in the trace monoid (in particular, such a
$w'$, when it exists, is unique). The adjoint of a creation operator is often called an \emph{annihilation operator}.

In a $\Gamma$-independent semicircular system $(S_u)_{u\in V}$, each $S_u$ has a semicircular distribution, in the sense that, for each $p \geq 1$, $\tau(S_u^{2p})$ is the $p$-th Catalan number and $\tau(S_u^{2p + 1}) = 0$. Furthermore, $S_u$ and $S_v$ are classically independent when $u$ and $v$ are adjacent in $\Gamma$, and freely independent otherwise.
The generalized notion of independence satisfied by these operators is called \emph{$\Gamma$-independence}~\cite{gammaIndependenceDefinition,CGVvH24}.\footnote{The same notion has also been referred to as $\varepsilon$-free independence~\cite{charlesworthCollins}, $\Lambda$-freeness~\cite{lambdaFreeness}, or $\Gamma$-right-angled independence~\cite{mageeThomas} in prior works.}
See Lemma~\ref{lem:momentsGamma} for a general formula for computing the joint moments of a $\Gamma$-independent semicircular system.

The graphs $\Gamma$ of interest in this paper are the {Kneser graphs}.

\begin{definition}\label{def:kneserGraph}
    For $\ell,r\ge 1$, the {\em Kneser graph} $K(\ell,r)$
    is the graph on the vertex set $\binom {[\ell]} r$, with an edge between $a$ and $b$ when $a\cap b=\varnothing$. For $k\ge 1$, let $K^{(k)}(\ell,r)$ denote the graph with vertex set $[k]\times \binom {[\ell]} r$, where $(i,a)$ and $(j,b)$ are adjacent when $a\cap b=\varnothing$.
\end{definition}

\subsection{Our results}

\paragraph{The $\Gamma$-independent limit.} In this paper, we identify the limit of Kikuchi matrices: it is given by a $\Gamma$-independent semicircular system, where $\Gamma$ is a Kneser graph.
Our first main result establishes 
weak convergence
of tuples of independent Kikuchi matrices.

\begin{theorem}\label{thm:weakIntro}
    Fix $k\ge 1$ and $1\le r\le \ell$. Let $\{S_{i,a} : i\in [k],\, a\in \binom {[\ell]} r\}$ be
    a $K^{(k)}(\ell,r)$-independent semicircular system. Define
    \begin{equation}
        \Sigma_i = \Sigma_i(\ell,r) \defeq {\binom \ell r}^{-1/2} \sum_{a\in \binom {[\ell]} r} S_{i,a}\qquad \forall i\in [k]\,.\label{eq:limitOperator}
    \end{equation}
    For each $n\ge \ell+r$, let ${\bd Y_{\!n}} = (Y_{n,1}, \ldots, Y_{n,k})$ be a $k$-tuple of independent copies of the Kikuchi matrix $Y_n(\ell,r)$ defined in~\eqref{eq:KikuchiYSTconstruction}.
    Then $\bd Y_{\!n}$ converges weakly as
    $n\to\infty$ to $\bd \Sigma = (\Sigma_1, \ldots, \Sigma_k)$.
\end{theorem}

This determines the asymptotic spectral 
distribution of Kikuchi matrices. 
Its moments have
a simple explicit form; see Lemma~\ref{lem:momentsGamma}.
As a special case, when $k=1$ and $\ell<2r$, the Kneser
graph $K(\ell,r)$ has no edges, so the limit
is semicircular, recovering the result of~\cite{BCSvH24}.

The basic reason to expect the Kneser graph to be involved is as follows.
In the combinatorics of tracial moments of a Kikuchi matrix, we sum over closed walks of a certain length on the \emph{Kikuchi graph} given by the pattern of non-zero entries.
Such walks can be viewed as moving around on $\binom{[n]}{\ell}$ by replacing $r$ indices of a subset at a time.
The leading-order contributions to the tracial moments come from walks in which each Gaussian variable (i.e., each total set of $2r$ indices affected by a step) occurs exactly twice, and such that the first step of each such pair introduces $r$ new indices that have not yet been encountered on the walk.
Further, these replacements can be viewed as occurring in $\ell$ labeled ``slots'' of an element of $\binom{[n]}{\ell}$, whose labels persist across replacements of indices.
In total, we then have a sequence of replacements of $r$ indices, ordered in a cycle, each replacement occurring in some $r$ slots out of $\ell$, where each replacement is matched to another that occurs on the same set of slots.
The main point is that two matches can only ``cross'' along the cycle if they occur on disjoint sets of slots.
This is precisely the adjacency relation in the Kneser graph, whereby the associated $\Gamma$-independent system has moments matching the leading-order contributions to tracial moments. See
also Section~\ref{sec:relatedWork} for
another interpretation of the limit in
terms of intrinsic freeness.

\paragraph*{}
Given Theorem~\ref{thm:weakIntro}, a
natural follow-up question is:

\begin{center}
    \textit{Under which 
    assumptions on $\ell,r$ can weak
    convergence be  upgraded to strong
    convergence? \\In particular, when does $\|Y_n\|\to \|\Sigma\|$?}
\end{center}
We are unable to answer this question in full, but we prove several partial results.
Note that weak convergence
implies $\|Y_n\|\ge \|\Sigma\|-o(1)$
with high probability, so the question is whether $Y_n$ can asymptotically have outlier eigenvalues lying beyond the support of its limiting spectral distribution.

To put our results in context, we first estimate $\|\Sigma\|$ as a function
of $\ell$ and $r$. Counting the contribution of words adding $p\approx \ell/r^2$
pairwise disjoint subsets of $[\ell]$ of size $r$ and removing them in an
arbitrary order gives the lower bound
\begin{equation}
    \|\Sigma\|\ge \tau(\Sigma^{2p})^{\frac 1 {2p}}\ge \left(p! \prod_{i=0}^{p-1} \frac {\binom {\ell-ir} r} {\binom \ell r} \right)^{\frac 1 {2p}} \gtrsim \frac {\sqrt{\ell}} r\,.\label{eq:lowerBoundNorm}
\end{equation}
We later show (Proposition~\ref{prop:upperBoundKneser}) that in fact, $\|\Sigma\|\asymp \max\{1,\sqrt \ell/r\}$.
This is compatible with the recent result of Kothari and Xu~\cite{kothari2025smoothtradeofftensorpca}, which confirms that $\|Y_n\|\asymp_r \sqrt \ell$ with high probability in the regime where
$r$ is constant. 
Our second main theorem is a refinement of~\cite{kothari2025smoothtradeofftensorpca} with
an explicit dependence on $r$ and an
improved allowed range of $\ell=\ell(n)$.

\begin{theorem}\label{thm:kothariXuImprovement}
    If $\ell\lesssim n/(r^2\log^2 n)$, then 
    $\|Y_n\|\lesssim \sqrt{\ell/r}$ with high probability.
\end{theorem}

Together with the lower bound~\eqref{eq:lowerBoundNorm} and weak convergence,
this determines the limit of $\|Y_n\|$
for fixed $\ell,r$
up to a factor $O(\sqrt r)$.

Next, we show that the norm convergence
question can be settled entirely
for a slight
variant of Kikuchi matrices. When $\ell=r=1$, we note that $Y_n$ is (up to its diagonal) just a normalized real symmetric matrix drawn from the \emph{Gaussian orthogonal ensemble} (GOE); we define the {\em directed Kikuchi matrix} to be the corresponding generalization of the complex Hermitian \emph{Gaussian unitary ensemble} (GUE).
Formally, given integers $n,\ell,r$ satisfying
$r\le \ell\le n-r$, the 
directed Kikuchi matrix $X_n=X_n(\ell,r)$ 
is the random matrix  with complex entries
\[
    X_n[U,V] \defeq \begin{cases}  \dfrac 1 {\sqrt{d_n}} \;g_{V\setminus U, U\setminus V} & \text{if $|U\mathbin \triangle V|=2r$}\\
    0 & \text{otherwise}
    \end{cases}\qquad \text{for }U,V\in \binom {[n]} \ell\,.
\]
Here, $g_{H,T}$ are standard complex Gaussians for all disjoint $H,T\subseteq [n]$ of size $|H|=|T|=r$, where $g_{H,T}$ and $g_{H',T'}$
are independent for every $\{H,T\}\neq \{H',T'\}$, and
$g_{H,T} = \overline{g_{T,H}}$ (complex conjugate). The normalization 
$d_n=d_n(\ell,r)$ is defined in~\eqref{eq:degreeKikuchi}.

Directed Kikuchi matrices have the same weak limit as Kikuchi matrices
(see Theorem~\ref{thm:weakConvergence}). Our third main result shows that,
for a single directed Kikuchi matrix, this weak convergence can be
upgraded to convergence of the norm.

\begin{theorem}\label{thm:directedIntro}
    For fixed $1\le r\le \ell$, we have $\left\|X_n\right\|\to \left\|\Sigma\right\|$ in probability as $n\to \infty$. 
\end{theorem}

The proof of Theorem~\ref{thm:directedIntro} generalizes to $\ell$ growing with $n$,
in the sense that $\|X_n\|\le (1+o(1)) \|\Sigma\|$ continues holding for constant $r$ and $\ell(n)\le n^{\delta}$
for some small constant $\delta>0$.

\paragraph{The $q$-Gaussian limit.} After taking the limit $n\to\infty$,
a further limit $\ell\to\infty$ with $r=r(\ell)$ connects Kikuchi matrices with
another central object in noncommutative
probability: the {\em $q$-Gaussian system}~\cite{frischBourret,brownianBozejko,qGaussianSupport,speicherCLT,qGaussianReloaded}. To
define this object, observe that
the creation operators in Definition~\ref{def:gammaSystem} satisfy
\begin{equation}
    L_u^\ast L_v = \Gamma_{uv} L_v L_u^\ast + \one\{u=v\}\cdot I\,,\label{eq:toeplitz}
\end{equation}
where $ I$ is the identity operator and $\Gamma$ is the adjacency matrix
of the graph $\Gamma$. 
The relations~\eqref{eq:toeplitz}
for arbitrary fixed matrices $\Gamma$
(not necessarily $\{0,1\}$-valued) are known
as the {\em $\Gamma$-Toeplitz relations}. The $q$-Gaussian system with
$k$ variables is obtained analogously by constructing creation operators satisfying the Toeplitz 
relations with $\Gamma_{uv} \defeq q$ for all
$u,v\in [k]$, for
some constant $q\in [0,1)$. See Definition~\ref{def:qGaussianSystem} for a formal definition.

Our fourth main result shows that tuples of
independent Kikuchi matrices converge to a
$q$-Gaussian system in an appropriate double limit.

\begin{theorem}\label{thm:qGaussianConvergence}
    Fix $k\ge 1$ and $\alpha>0$. Let $\bd X_{\!n}$ and $\bd Y_{\!n}$ be $k$-tuples of independent directed Kikuchi and Kikuchi matrices, respectively.
    Then in the double limit $n\to\infty$,
    followed by $\ell\to\infty$ with $r=r(\ell)$ such that $r^2/\ell\to \alpha$,
    both $\bd X_{\!n}$ and $\bd Y_{\!n}$ converge weakly to the $e^{-\alpha}$-Gaussian system with $k$ variables.
    Furthermore, for $k=1$,
    $X_{\!n}$ converges strongly to the same
    limit.
\end{theorem}

In light of the multivariate central limit theorem for $\Gamma$-independent random variables
of~\cite{graphonCLT1}, the weak convergence statement in
Theorem~\ref{thm:qGaussianConvergence} is equivalent to showing that
the blowups of the Kneser graphs converge to the constant graphon. We conjecture that the strong convergence 
statement for the directed Kikuchi matrix
extends to any $k\ge 1$ and $\alpha>0$
(for $\alpha=0$, the linear operators
of the corresponding ``$1$-Gaussian system'' are unbounded, so strong convergence does not make sense).

\subsection{Connections and broader context}\label{sec:relatedWork}

\paragraph{Intrinsic $\Gamma$-independence.}
The Kikuchi matrix is a Gaussian matrix series: in the notation of~\eqref{eq:KikuchiYSTconstruction},
\[
    Y_n = \sum_{I\in \binom {[n]} {2r}} g_I M_I\,,\qquad \text{where } M_I[U,V] \defeq \begin{cases}\dfrac 1 {\sqrt {d_n}} & \text{if $U\mathbin \triangle V = I$},\\
    0 & \text{otherwise.}\end{cases}
\]
Given this representation, the intrinsic freeness approach of~\cite{BBvH23} suggests comparing
the spectrum of $Y_n$ to that of the operator obtained by replacing the independent Gaussians with freely independent semicircular elements:
\begin{equation}
    Y_{\text{free}} \defeq \sum_{I\in \binom {[n]} {2r}} S_I \otimes M_I\,,\qquad \text{$\{S_I\}$ is a freely independent semicircular system.}\label{eq:free}
\end{equation}
However, as discussed above, the spectrum of $Y_{\text{free}}$ does not
asymptotically match that of $Y_n$ once $\ell\ge 2r$. Instead, our results suggest that the appropriate operator to compare
$Y_n$ to is
\begin{equation}
    Y_{\text{$\Gamma$}} \defeq \sum_{I\in \binom {[n]} {2r}} S_I \otimes M_I\,,\qquad \text{$\{S_I\}$ is a $K(n,2r)$-independent semicircular system.}\label{eq:gammaFree}
\end{equation}
Indeed, for every $\ell,n,r$, the operator $Y_{\Gamma}$ has
{exactly} the same moments as the normalized sum of the $\Gamma$-independent
semicircular system associated with
the graph product
$\Gamma = K(\ell,r)\times K(n-\ell,r)$.\footnote{That is, the graph with vertex
set $\binom {[\ell]} r\times \binom {[n-\ell]} r$, with an edge between $(a,a')$ and $(b,b')$ if and only if $a\cap b=a'\cap b'=\varnothing$.}
As $n\to\infty$ for fixed
$\ell,r$, the second factor $K(n-\ell,r)$ 
approaches a complete graph, and therefore $Y_{\Gamma}$ converges weakly to
the normalized sum of the $K(\ell,r)$-independent semicircular system -- precisely the weak limit of $Y_n$ identified in Theorem~\ref{thm:weakIntro}.

Thus, the example of Kikuchi matrices motivates
the development of a 
theory of ``intrinsic $\Gamma$-independence''.
As a first step in this direction,
the parallel work~\cite{BKNPW2026-Paper2} derives general
upper bounds on the norms of operators of the form~\eqref{eq:gammaFree}, in the spirit of what Lehner's formula achieves for~\eqref{eq:free} (see Proposition~\ref{prop:pseudoRandomUpperBound}).

\paragraph{Hypergraph Moore bound.}
The dense Kikuchi matrix~\eqref{eq:KikuchiYSTconstruction} has a natural sparse generalization: Given any $2r$-uniform
hypergraph $\calH$ on $n$ vertices,
define $Y_\calH$ by keeping only the entries $(U,V)$ of $Y$
such that $U\mathbin\triangle V$ is
a hyperedge of $\calH$ (and setting all
other entries to $0$).
Bounding the norm of these random
sparse Kikuchi matrices $Y_\calH$ was an instrumental step in the resolution, up to logarithmic factors, of Feige's conjecture on the optimal girth-density tradeoff of hypergraphs, also known as the hypergraph Moore bound~\cite{GuruswamiKothariManohar2022,HsiehKothariMohanty2022simpleMoore}.
One of the sources of the spurious logarithmic factor is the application of bounds such as~\eqref{eq:usingNCK}. 

Our initial motivation to study limits of Kikuchi matrices was to remove these remaining logarithmic factors to establish the hypergraph Moore bound. Indeed, we were able to show that a natural candidate limiting object does have the spectral properties that would be needed to establish the conjecture. This appears as an application of a method we developed to understand objects arising in non-commutative probability~\cite{BKNPW2026-Paper2}. However, we were not able to show the level of convergence needed to establish the hypergraph Moore bound via this route. 

While this work was in progress, the hypergraph Moore bound was established in two independent works~\cite{BKNPW2026-Moore,schmidhuberHastings2026}, using different techniques.

\paragraph{Matrix models for noncommutative independence.} 
We show that $\Gamma$-independent and $q$-Gaussian systems appear naturally as limits of
Kikuchi matrices. 
Historically, significant attention has 
been devoted to the converse direction:
starting from a prescribed set of noncommutative relations, can one
construct ``natural'' sequences of random matrices that realize them
asymptotically, in the sense of weak or strong convergence? 

The question of constructing such a random matrix
model for the $q$-Gaussian system has its origins in mathematical physics, going back to
Frisch and Bourret in 1970~\cite{frischBourret}, and was solved
in the weak sense in~\cite{speicherCLT,sniadyModel}.
For $\Gamma$-independence, weak matrix models were
constructed in~\cite{charlesworthCollins}
and strong matrix models in~\cite{mageeThomas,CGVvH24}. 
One example is the 
{\em tensor GUE} model.
For every fixed $k$, the results of~\cite{freeMonotoneTransport}, 
combined with strong convergence of the GUE~\cite{HT05}, 
yield a strong matrix model
for $q$-Gaussian systems with $k$ variables when $q$ is smaller than a positive constant depending on $k$.
Very recently, strong matrix models were constructed in the range $q<\sqrt 2-1$, for any $k$~\cite{strongQGaussian}; the case of larger $q$ remains open.\footnote{The construction in~\cite{strongQGaussian} has some similarities with ours.
The authors consider
random matrices converging strongly to a $\Gamma_\ell$-independent semicircular system for each $\ell$, where $\Gamma_\ell$ is a sequence of \Erdos-\Renyi random graphs. A further limit $\ell\to\infty$ combined with a diagonal argument yields their $q$-Gaussian matrix model.} 

In contrast, we obtain weak (and conjecturally strong) models for these systems from the opposite direction: they arise naturally in the analysis of efficient algorithms for tensor PCA.

\subsection{Organization of the paper}

We gather preliminaries for the technical sections in Section~\ref{sec:prelims}.
Each subsequent section focuses
on one of our main theorems and can be read largely
independently:
Section~\ref{sec:weak} proves
Theorem~\ref{thm:weakIntro}; Section~\ref{sec:norm} proves Theorem~\ref{thm:kothariXuImprovement}; Section~\ref{sec:strongConvergence} proves Theorem~\ref{thm:directedIntro}; and 
Section~\ref{sec:qGaussian} proves Theorem~\ref{thm:qGaussianConvergence}.

\subsection{Acknowledgments and AI disclosure}

LP acknowledges support from an ETH Z\"urich Postdoctoral Fellowship (Project No.\ 26-1 FEL-092).
Part of this work was done while RW was an academic guest in the Department of Mathematics at ETH Z\"urich. RW's work on this project was supported by the Queen Elizabeth II Graduate Scholarship in Science and Technology (QEII-GSST).
PNN and RW thank the Simons Institute for the Theory of Computing and the organizers of the Fall 2026 program on \emph{Spectral Theory Beyond Graphs}.

Most of this project was developed before AI tools could help with mathematics as they can now. Nevertheless, more recently, we have used AI tools to help perform some computations, carry through some specific arguments, and help with references. The authors wrote the paper and take responsibility for all its (potential) errors.
\section{Preliminaries}
\label{sec:prelims}

All graphs in this paper are simple and undirected. For a graph
$\Gamma$, let $V(\Gamma)$, $E(\Gamma)$ denote its vertex and edge
sets, respectively; we also write $\Gamma$ for its adjacency matrix. We define
\[
    \Hom(G, \Gamma) \defeq \{\varphi : V(G)\to V(\Gamma) : \{\varphi(u),\varphi(v)\}\in E(\Gamma)\;\;\forall \{u,v\}\in E(G)\}
\]
to be the set of graph homomorphisms from a graph $G$ to $\Gamma$.
Let $\calP(p)$ be the set of {\em pairings} (perfect matchings) on 
$[p]$. We denote the edges of a pairing
$\pi$
by $ij\in E(\pi)$ when $1\le i<j\le p$. For each
$ij\in E(\pi)$, we
call $i$ an {\em opening index} and $j$ a {\em closing index}. 
Finally, let
\[
    \calP(\bd w)\defeq \{\pi \in \calP(p) : w_i = w_j \;\;\forall\, ij \in E(\pi)\}\qquad \text{ for any } \bd w = (w_1, \ldots, w_p)\,.
\]
Pairings appear naturally
when computing the moments of multivariate Gaussians:

\begin{lemma}[Wick's formula]\label{lem:wick}
    If $g\in \C^p$ is a centered Gaussian vector, then
    \[
        \E \left[\prod_{i=1}^p g_i\right] = \sum_{\pi \in \calP(p)} \prod_{ij\in E(\pi)} \E\left[g_i g_j\right]\,.
    \]
\end{lemma}

\subsection{State and operator norm}\label{sec:operatorNotation}

Let $\calB(\calH)$ be the space
of bounded linear maps $\calH\to \calH$
on a Hilbert space $\calH$. Let $I$ be the identity operator. We write $A^\ast$ for the
adjoint of the operator $A$, and $\|A\|$
for the operator norm of $A$.

The moments of an operator are
specified by a {\em state} $\tau:\calB(\calH)\to \C$,
which is a linear functional such that $\tau(I)=1$ and $\tau(A^\ast A)\ge 0$ for all
$A\in \calB(\calH)$. 
The definition automatically implies
\begin{equation}
    \tau\left((A^\ast A)^p\right)^{\frac 1 {2p}}\le \|A\|\qquad \forall p\ge 1\,.\label{eq:faithfulAutomatic}
\end{equation}

Given self-adjoint operators $A_1, \ldots, A_k\in \calB(\calH)$, let $C^*(A_1, \ldots, A_k)$ be the  closure (in the
operator norm topology)
of the set of all noncommutative polynomials in $A_1, \ldots, A_k$.
We say that $\tau$
is {\em faithful} on $C^*(A_1, \ldots, A_k)$ if
$\tau(A^\ast A)\neq 0$ for all $A\in C^*(A_1, \ldots, A_k)\setminus \{0\}$.
Under that assumption, the norm of an operator can be
recovered from its moments.

\begin{lemma}[{\cite[Proposition 3.17]{NS06}}]\label{lem:faithful}
    Let $A_1, \ldots, A_k$ be self-adjoint
    operators. If $\tau$ is faithful on $C^\ast(A_1, \ldots, A_k)$, then
    \[
        \lim_{p\to\infty} \tau((A^\ast A)^p)^{\frac 1 {2p}} = \|A\|\qquad \forall A\in C^\ast(A_1, \ldots, A_k)\,.
    \]
\end{lemma}

Finally, we say that $\tau$ is {\em tracial} on
$C^*(A_1, \ldots, A_k)$ if $\tau(AB) = \tau(BA)$ for
all $A,B\in C^*(A_1, \ldots, A_k)$. 

\begin{lemma}[Hölder's inequality; {see~\cite[Section 1]{pisierXu}}]\label{lem:holder}
    Let $A_1, \ldots, A_k$ be self-adjoint
    operators. If $\tau$ is tracial on $C^\ast(A_1, \ldots, A_k)$ and $k$ is even, then
    \[
        \left|\tau(A_1\cdots A_k\right)|\le \tau(A_1^k)^{\frac 1 {k}} \cdots \tau(A_k^k)^{\frac 1 {k}}\,.
    \]
\end{lemma}

\subsection{\texorpdfstring{$\Gamma$-independent and $q$-Gaussian systems}{Gamma-independent and q-Gaussian systems}} 

The following theorem from~\cite{proofGammaRepresentation} ensures the existence of operators satisfying the $\Gamma$-Toeplitz 
relations.

\begin{theorem}\label{thm:existenceToeplitz}
    Let $\Gamma\in [-1,1]^{k\times k}$ be a symmetric matrix such that $\Gamma_{ii}<1$ for all $i\in [k]$. Then there exists a Hilbert space $\calH$, a distinguished unit vector
    $x\in \calH$ and  operators $L_1, \ldots, L_k\in \calB(\calH)$ with the following properties. Letting $\tau(A) \defeq \iprod{x, A x}$ and $S_i \defeq L_i+L_i^\ast$, we have that
    \begin{enumerate}
        \item $L_1, \ldots, L_k$
        satisfy the $\Gamma$-Toeplitz relations~\eqref{eq:toeplitz};
    
        \item $L_i^\ast x = 0$ for
        each $i\in [k]$;

        \item $\tau$ is faithful and tracial on $C^\ast(S_1, \ldots, S_k)$.
    \end{enumerate}
\end{theorem}

\begin{remark} 
    Although the concrete Hilbert space representation need not be unique, conditions 1-2-3 of
    Theorem~\ref{thm:existenceToeplitz}  uniquely determine the information needed for weak and strong convergence
    to $(S_1, \ldots, S_k)$.
    Indeed, it is not hard to see that conditions 1 and 2 uniquely determine by induction the joint moments  $\tau(P(S_1, \ldots, S_k))$
    for every noncommutative polynomial $P$. Condition 3 then implies by Lemma~\ref{lem:faithful} that
    these moments also determine the norm of every such polynomial.
\end{remark}

Definitions~\ref{def:traceMonoid} and~\ref{def:gammaSystem} follow a 
known construction
for proving Theorem~\ref{thm:existenceToeplitz} in the special case where $\Gamma$ is the adjacency matrix of a graph~\cite{proofGammaRepresentation,youssefTuran}. In this case,
$\calH$ is the Hilbert space indexed by the elements of the trace monoid,
$L_1,\ldots,L_k$ are the corresponding creation operators, and $x$ is the
basis vector corresponding to the empty word.

\begin{definition}\label{def:qGaussianSystem}
    Let $q\in [0,1)$. A {\em $q$-Gaussian system} on $k$ variables
    is a tuple  $(S_1, \ldots, S_k)$
    satisfying the conclusion of  Theorem~\ref{thm:existenceToeplitz}
    for $\Gamma_{ij} = q$ for all $i,j\in [k]$. Each $S_i$ is called a {\em $q$-Gaussian variable}.
\end{definition}

Finally, we compute the moments of $\Gamma$-independent and $q$-Gaussian systems explicitly. The {\em crossing graph}, $G_\pi$, 
of a pairing $\pi$ is the graph whose vertices
$V(G_\pi) = E(\pi)$ are edges of $\pi$, and
two vertices in $G_\pi$ are adjacent if
the corresponding edges form a crossing in $\pi$. Two edges $uu',vv'\in E(\pi)$ are {\em crossing} if $u<v<u'<v'$ or $v<u<v'<u'$.

\begin{lemma}[See {\cite[Proposition 2]{brownianBozejko}}]\label{lem:momentsQGaussian}
    Let $(S_1, \ldots, S_k)$ be a $q$-Gaussian system. Then 
    \[
        \tau(S_{w_1}\cdots S_{w_p}) = \sum\limits_{\substack{\pi\in \calP(\bd w)}} q^{|E(G_\pi)|}\qquad \forall p\ge 1\quad \forall \bd w\in [k]^p\,.
    \]
\end{lemma}

\begin{lemma}[See {\cite[Theorem 5.2]{gammaIndependenceDefinition}}]\label{lem:momentsGamma}
    Let $(S_u)_{u\in V(\Gamma)}$ be a $\Gamma$-independent semicircular system.
    Then for any $p\ge 1$ and any choice of complex coefficients $\{c_{i,u} : i \in [p], u\in V(\Gamma)\}$,
    \[
        \tau\left(\left(\sum_{u\in V(\Gamma)} c_{1,u} S_u\right)\cdots \left(\sum_{u\in V(\Gamma)} c_{p,u} S_u\right)\right) = \sum_{\substack{\pi \in \calP(p)\\\varphi \in \Hom(G_\pi, \Gamma)}}  \prod_{ij\in E(\pi)} c_{i,\varphi(ij)} c_{j,\varphi(ij)}\,.
    \]
    In particular,
    \begin{equation*}
        \tau\left(\left(\sum_{u\in V(\Gamma)} S_u\right)^p\right) = \sum_{\pi \in \calP(p)} \left|\Hom(G_\pi,\Gamma)\right|\,.
    \end{equation*}
\end{lemma}

\subsection{Kikuchi walks}

In this subsection, we start introducing the
common terminology for the technical 
sections. 
In particular, we define {\em Kikuchi walks}, which are the combinatorial objects that arise in the trace moment expansion of Kikuchi matrices.

\begin{definition}
    A {\em directed hyperedge} is a pair 
    \[
        e = (H(e),T(e))\in {\binom {[n]} r}\times {\binom {[n]} r}\quad\text{such that $H(e)\cap T(e)=\varnothing$}\,.
    \]
    The sets $H(e)$ and $T(e)$ are called the {\em head} and {\em tail} of $e$, respectively. The {\em support} of
    $e$ is $\supp(e) \defeq H(e)\cup T(e)$.
    The {\em reverse} of $e$ is $e^\ast \defeq (T(e), H(e))$. Let $\calE_{n,r}$ be the set of all directed hyperedges.
\end{definition}

\begin{definition}
A {\em Kikuchi walk} of length $p$ is a tuple
$\bd e = (e_1,\ldots, e_{p})\in \calE_{n,r}^p$
for which
\begin{equation*}
    U_0=[\ell]\,,\qquad U_i \defeq \left(U_{i-1}\cup H(e_i)\right)\setminus T(e_i)\,,
\end{equation*}
satisfy
$U_p=U_0$, $T(e_i)\subseteq U_{i-1}$ and $H(e_i)\cap U_{i-1} = \varnothing$ for all $i\in [p]$.
\end{definition}

Define the support sequence of a Kikuchi walk by $\supp(\bd e) \defeq (\supp(e_i))_{i\in [p]}$. We call a Kikuchi walk {\em even} if every support appearing in its
support sequence occurs an even number of times, and denote by $\calW_p$ the set of even Kikuchi
walks of length $p$. Note that even Kikuchi walks
automatically have even length.

To lighten notation, we will write $U_i$, $T_i\defeq T(e_i)$, $H_i \defeq H(e_i)$, where the dependence on $\bd e$ is implicit but should be clear from context.

\begin{lemma}\label{lem:traceMoments}
    The trace moments of the Kikuchi matrix $Y_n$ and
    directed Kikuchi matrix $X_n$ satisfy
    \begin{align}
        \E \tr Y_n^{p} &= {d_n^{-\frac p2}} \left|\left\{(\bd e,\pi) \in \calW_p \times \calP(p) :  \supp(e_i)=\supp(e_j) \;\;\forall ij \in E(\pi)\right\}\right|\,,\label{eq:exactExpansionKikuchi}\\
        \E \tr X_n^{p} &= {d_n^{-\frac p2}} \left|\left\{(\bd e,\pi) \in \calW_p \times \calP(p) :  e_i = e_j^\ast \;\;\forall ij \in E(\pi)\right\}\right|\,.\label{eq:exactExpansionDirected}
    \end{align}
\end{lemma}

\begin{proof}
    This is a direct application
    of Wick's formula (Lemma~\ref{lem:wick}), together with the invariance of $X_n$ and $Y_n$ by relabeling (which allows us to average out the choice of $U_0$).
\end{proof}

We organize the Kikuchi walks that contribute to~\eqref{eq:exactExpansionKikuchi} and~\eqref{eq:exactExpansionDirected} in terms of the
number of vertices that they introduce. For $i\in [p]$, let \[R_i \defeq H_i\setminus \left(\bigcup_{j=0}^{i-1} U_j\right)\]
be the set of vertices introduced for the first time at step $i$.
For an even walk $\bd e$, every support appearing in $\supp(\bd e)$ occurs at least twice, a support
that has appeared previously cannot introduce any
new vertex, and any step can introduce at most $r$ vertices. Therefore, an even walk can introduce $\sum_{i=1}^p |R_i|\le rp/2$
vertices in total. 
We call the integer
$rp/2 - \sum_{i=1}^p |R_i|$ the {\em defect}
of the walk $\bd e$. 

\subsection{Kikuchi transports}\label{sec:transports}

By Lemma~\ref{lem:momentsGamma}, the
combinatorial objects
that contribute to the
moments of the $K(\ell,r)$-independent semicircular system 
are tuples $(\varphi,\pi)$ where
$\pi \in \calP(p)$ and
$\varphi \in \Hom(G_\pi,K(\ell,r))$.
In words, $\varphi$ is an assignment of elements of $\binom \ell r$ to
every pair in $\pi$, with the property that these subsets are
disjoint whenever two pairs cross.
To map the paired Kikuchi walks from
Lemma~\ref{lem:traceMoments} to
these predicted limiting objects,
we introduce the notion
of {\em Kikuchi transport}.

\begin{definition}\label{def:transport}
    Let $m\ge 0$ and $(\bd e, \pi)\in \calW_p\times \calP(p)$.
    A \emph{Kikuchi transport} of
    $(\bd e,\pi)$ is a sequence $\bd \rho=(\rho_0,\ldots,\rho_p)$ of injective maps $\rho_i:U_i\to[\ell+m]$ with
    $\rho_0=\textnormal{id}_{[\ell]}$, such that
    \begin{itemize}
        \item {(Pairing compatibility)}
        $\rho_{i-1}(T_i) = \rho_{j-1}(T_j)$ for any $ij\in E(\pi)$.
    
        \item {(Crossing compatibility)}
        $\rho_{i-1}(T_i)
        \cap \rho_{i'-1}(T_{i'})
        = \varnothing$ for any $\{ij,i'j'\}\in E(G_\pi)$.
    \end{itemize}
    We call $m$ the {\em excess size} of $\bd \rho$.
\end{definition}

\noindent Equivalently, a Kikuchi transport is a sequence of embeddings such that the map
\begin{equation}\label{eq:transport_homomorphism}
    \begin{aligned}
        G_\pi &\longrightarrow K(\ell+m,r)\\
        ij &\longmapsto \rho_{i-1}(T_i) = \rho_{j-1}(T_j)
    \end{aligned}
\end{equation}
is a graph homomorphism. By Lemma~\ref{lem:momentsGamma}, these are exactly the
objects that moments of Kneser
semicircular systems count (for $m=0$).

When designing Kikuchi transports,
we will try to make sure that they respect as much as possible the
{\em canonical bijections} $\phi_t : U_{t}\to U_{t-1}$, which are the maps fixing $U_t\cap U_{t-1} = U_{t-1}\setminus T_t = U_t\setminus H_t$, and mapping the
elements of $ H_t=U_t\setminus U_{t-1} $ to the
elements of $T_{t}=U_{t-1}\setminus U_t$ in
increasing order.

\begin{definition}\label{def:switches}
    Let $\bd \rho$ be a Kikuchi
    transport of $(\bd e,\pi)$.
    A pair $(t,u)\in [p]\times U_t$ is called a {\em switch} if $\rho_t(u)\neq \rho_{t-1}(\phi_{t}(u))$. The {\em cost} of
    $\bd \rho$ is the total number of
    switches $(t,u)$ it incurs.
\end{definition}
\section{Weak Convergence of Kikuchi Matrices} 
\label{sec:weak}

In this section, we prove weak convergence
of Kikuchi matrices  (Theorem~\ref{thm:weakIntro})
and directed Kikuchi matrices to their
$\Gamma$-independent limit as special
cases of Theorem~\ref{thm:weakConvergence} below.

\begin{definition}\label{def:generalizedKikuchi}
    Let $\ell,r,n$ be integers such that $1\le r\le \ell \le n-r$. Let $d_n=d_n(\ell,r)$ be as in~\eqref{eq:degreeKikuchi}.
    A {\em generalized Kikuchi matrix} $Z_n = Z_n(\ell,r)$ is a matrix
    whose entries have the form
    \[
        Z_n[U,V] \defeq \begin{cases}\dfrac 1 {\sqrt{d_n}} \, {g_{(V\setminus U, U\setminus V)}}  & \text{ if $|U \mathbin \triangle V|=2r$}\\
        0 & \text{ otherwise}\end{cases}
        \quad \forall U,V\in \binom {[n]} \ell\,,
    \]
    where $(g_e)_{e\in \calE_{n,r}}$ is a
    centered (real or complex) Gaussian vector
    satisfying $g_{e^\ast} = \overline{g_e}$, $\E |g_e|^2=1$
    and $\E \left[g_e g_f\right]=0$ whenever $\supp(e)\neq \supp(f)$.
\end{definition}

\noindent The Kikuchi matrix is obtained
by sampling a collection of i.i.d.\
standard real Gaussians $g'_S$ for $S\subseteq [n]$ of size $|S|=2r$ and setting $g_{e} = g'_{\supp(e)}$.
The directed Kikuchi matrix is obtained by sampling
$(g_e)_{e\in\calE_{n,r}}$ as standard complex Gaussians,
independently up to the symmetry
$g_{e^\ast}=\overline{g_e}$.

\begin{theorem}\label{thm:weakConvergence}
    Fix $k\ge 1$ and $1\le r\le \ell$. Let $\bd \Sigma = (\Sigma_1, \ldots, \Sigma_k)$ be as in~\eqref{eq:limitOperator}.
    Any sequence of $k$-tuples of independent generalized Kikuchi matrices $\bd Z_n = (Z_{n,1}, \ldots, Z_{n,k})$ converges weakly
    as $n\to\infty$ to $\bd \Sigma$.
\end{theorem}

The proofs in this section and in Sections~\ref{sec:norm} and~\ref{sec:strongConvergence}
are all based on counting 
Kikuchi walks and compatible pairings
appearing in the expansion of the trace
moments. The difference between the
arguments lies in how the walks are counted.
For weak convergence, the moment $p$ is fixed
as $n\to\infty$. In this regime, the dominant
contributions come from the Kikuchi walks
of defect $0$. Since $p$ is fixed, it then
suffices to use crude
bounds on the number of paired
walks of higher defect
(Lemma~\ref{lem:boundPairingGenus}).

\begin{lemma}\label{lem:genusZero}
    An even Kikuchi walk of length $p$ has defect $0$ if and only
    if $\supp(\bd e)$ consists of $p/2$ supports, each occurring twice, and
    the first occurrence of each support introduces
    exactly $r$ vertices.
\end{lemma}

\begin{proof}
    For a defect-$0$ walk, the above argument that bounds the number of vertices introduced in the
    walk must hold with equality, so every support
    in $\supp(\bd e)$ occurs exactly twice and every 
    first occurrence of a support introduces exactly $r$ vertices.
    The converse statement is immediate.
\end{proof}

Lemma~\ref{lem:genusZero} shows that a defect-$0$ walk $\bd e$ determines a unique pairing
$\pi(\bd e)\in\calP(\supp(\bd e))$. This process pairs
together reverse hyperedges:

\begin{lemma}\label{lem:zeroGenusReverse}
    If $\bd e$ has defect $0$, then $e_i=e_j^\ast$
    for every $ij\in E(\pi(\bd e))$.
\end{lemma}

\begin{proof}
    We proceed by induction on the
    length $p$ of $\bd e$. The statement holds for $p=0$.
    
    Let $p\ge 2$ and suppose that the statement holds for $p-2$. Let $j\in [p]$ be the smallest closing index, and let $ij\in E(\pi(\bd e))$. For every $i<k<j$,
    $e_k$ must be the first occurrence of some
    support by minimality of $j$, and by Lemma~\ref{lem:genusZero}, the vertices
    in $H_k$ are all new. This means that none
    of the vertices of $T_i$ that have been removed
    at step $i$ can be re-introduced before step $j$.
    Since $T_i\subseteq \supp(e_j)$ and $|T_i|=|H_j|=r$, it must be that $T_i=H_j$.
    Since heads and tails are disjoint and $\supp(e_i)=\supp(e_j)$ by assumption, it
    follows automatically that $e_i=e_j^\ast$.
    
    Finally, removing $e_i$ and $e_j$ from the walk yields a defect-$0$ walk of length $p-2$ by Lemma~\ref{lem:genusZero}. By induction, the property holds for any edge in $\pi(\bd e)$.
\end{proof}

\begin{lemma}\label{lem:countingGenusZero}
    Let $p$ be even and $\pi\in \calP(p)$. The number of even Kikuchi walks of length $p$, defect-$0$, and such that $\pi(\bd e)=\pi$ is
    \[
        \frac {(n-\ell)\cdots (n-\ell-\frac {rp} 2+1)} {(r!)^{\frac p 2}} |\Hom(G_{\pi},K(\ell,r))|\,.
    \]
\end{lemma}

\begin{lemma}\label{lem:transportZero}
    For any defect-$0$ even Kikuchi walk $\bd e$,
    $(\bd e,\pi(\bd e))$ has a Kikuchi transport of excess size $0$ and cost $0$.  
\end{lemma}

\begin{proof}[Proof of Lemma~\ref{lem:countingGenusZero} from Lemma~\ref{lem:transportZero}]
    The map~\eqref{eq:transport_homomorphism} induced by the transport is a graph homomorphism $G_\pi \to K(\ell,r)$. Moreover, at
    each opening index, the head
    consists of $r$ new vertices.
    Thus, the $p/2$ opening heads
    are pairwise disjoint subsets
    of $[n]\setminus [\ell]$ of size
    $r$; the number of such objects
    is
    \[
        \binom {n-\ell} r\cdot \binom {n-\ell-r} r\cdots \binom {n-\ell-r(\frac p2-1)} r = \frac {(n-\ell)\cdots (n-\ell-\frac {rp} 2+1)} {(r!)^{\frac p 2}}\,.
    \]
    Conversely, a graph homomorphism together with the disjoint sets assigned to the opening heads uniquely determines the walk.
    Indeed, we can
    reconstruct the walk and a
    cost-$0$ transport of it simultaneously from left to right. Suppose that we have
    reconstructed $\rho_1,\ldots,\rho_{i-1}$ and $U_1, \ldots, U_{i-1}$. If $i$ is a closing
    index, then the head and tail
    of $e_i$
    are determined from the matching opening by the
    mirroring property of Lemma~\ref{lem:zeroGenusReverse}.
    If $i$ is an opening index with $ij\in E(\pi)$, then $H_i$ is prescribed by the opening
    head data, and $T_i = \rho^{-1}_{i-1}(\varphi(ij))$. This determines $U_i$, and 
    the cost-$0$ condition then uniquely determines $\rho_{i}$ from $\rho_{i-1}$.
\end{proof}

\begin{proof}[Proof of Lemma~\ref{lem:transportZero}]
    Define recursively $\rho_0(u) = u$, $\rho_t(u) = \rho_{t-1}(\phi_t(u))$. Then $\rho_t$ is automatically injective, has image $[\ell]$, and incurs $0$ switches. 
    
    We show by induction on $p$ that $\bd \rho$ is a
     Kikuchi transport. 
    Let $ij\in E(\pi)$ be such that
    $j$ is the smallest closing index. Then every
    index in $\{i+1,\ldots,j-1\}$
    is an opening index. None of these indices can remove a vertex of $H_i=T_j$: indeed, once removed, such a vertex could not be reintroduced before time $j$, since every intermediate head consists only of new vertices (which would
    contradict $T_j\subseteq U_{j-1}$). Therefore, the labels of $H_i$ remain
    unchanged through the canonical bijections, and $\rho_{j-1}(T_j) = \rho_i(H_i) = \rho_{i-1}(T_i)$.
    Next, suppose that $\{ij,i'j'\}\in E(G_\pi)$ (note that $i'>i$ by minimality of $j$). By the above
    argument, we have $T_{i'}\cap H_i=\varnothing$ and $\rho_{i-1}(T_i)=\rho_{i}(H_i) = \rho_{i'-1}(H_i)$, so by injectivity of $\rho_{i'-1}$, we get $\rho_{i'-1}(T_{i'})\cap \rho_{i-1}(T_i) = \varnothing$. This shows that
    the conditions in Definition~\ref{def:transport} hold
    for $ij$; removing these indices
    from the walk and the pairing
    yields a defect-$0$ walk of
    length $p-2$, which implies
    the statement by induction.
\end{proof}

Lemma~\ref{lem:genusZero} implies that defect-$0$ walks have exactly one pairing respecting
their support. We robustify this bound as
follows:

\begin{lemma}\label{lem:boundPairingGenus}
    For any even Kikuchi walk $\bd e$ of length $p$ and defect
    $\Delta$, $|\calP(\supp(\bd e))|\le p^{\Delta/r}$.
\end{lemma}

\begin{proof}
    Let $m$ be the number of distinct supports in $\supp(\bd e)$. Since each directed hyperedge can introduce at most $r$ new vertices and there are $rp/2-\Delta$ new vertices in total, we have
    $p/2-m\le \Delta/r$.
    Furthermore,
    if the supports have multiplicities $\beta_1, \ldots, \beta_m$,
    \[
        |\calP(\supp(\bd e))| = \prod_{j=1}^m (\beta_j-1)!!\le \prod_{j=1}^m \beta_j^{\beta_j/2-1}\le p^{p/2-m}\le p^{\Delta/r}\,,
    \]
    which concludes the proof.
\end{proof}

\newcommand\weight{\textnormal{wt}}

Let $\bd Z = (Z_{1}, \ldots, Z_{k})$ be a tuple of independent generalized Kikuchi matrices, where we drop the dependence on $n$ to lighten notation.
For $i\in [k]$, we denote by $(g_{i,e})_{e\in \calE_{n,r}}$
the Gaussian vector associated to $Z_{i}$.
Given a Kikuchi walk $\bd e$ of length $p$ and $\bd w\in [k]^p$,
define the weight
\begin{equation}
    \weight_{\bd Z}(\bd w,\bd e) \defeq \E \left[\prod_{a=1}^p g_{w_a,e_a}\right] = \sum_{\pi \in \calP(\bd w)} \prod_{ab\in E(\pi)} \E \left[g_{w_a,e_a} g_{w_b,e_b}\right]
    \,,\label{eq:defWeight}
\end{equation}
where the identity follows from Wick's formula (Lemma~\ref{lem:wick}) and independence of the matrices within the tuple. 
We make several simple observations about this definition of weight.
\begin{enumerate}
    \item Only pairings $\pi\in \calP(\supp(\bd e))$ can have
    nonzero contribution to the sum on the right-hand side of~\eqref{eq:defWeight}.
    In particular, only even Kikuchi walks can have nonzero weight.
    \item Since the Gaussians have variance
    $1$, Cauchy-Schwarz implies that every 
    term in the sum has magnitude at most $1$. It then follows from Lemma~\ref{lem:boundPairingGenus} that the weight
    of any defect-$\Delta$, length-$p$ walk is bounded by $|\calP(\supp(\bd e))|\le p^{\Delta/r}$.
    \item By Lemma~\ref{lem:genusZero}, if
    $\bd e$ has defect $0$, then $\calP(\supp(\bd e)) = \{\pi(\bd e)\}$, so
    combining with the first observation and Lemma~\ref{lem:zeroGenusReverse}, we have $\weight_{\bd Z}(\bd w,\bd e) = 1$ if $\pi(\bd e)\in \calP(\bd w)$, and $0$ otherwise.
\end{enumerate}

\begin{proof}[Proof of Theorem~\ref{thm:weakConvergence}]
    Our definition of a Kikuchi walk 
    assumes that its initial set is $[\ell]$. To reduce to this case, 
    for each $U\in\binom{[n]}{\ell}$, 
    fix an arbitrary permutation $\sigma_U$ of $[n]$ satisfying $\sigma_U([\ell])=U$, and let
    $\bd Z(U)=(Z_1(U),\ldots,Z_k(U))$     
    be the tuple obtained by 
    simultaneously permuting the 
    rows and columns of every $Z_i$ according to $\sigma_U$. Then $\bd Z(U)$ is again 
    a tuple of independent generalized Kikuchi matrices.

    For every
    $\bd w\in [k]^p$, we have
    \begin{equation}
        \E \tr \left(Z_{w_1}\cdots Z_{w_p}\right) = {\binom n \ell}^{-1} \sum_{U\in \binom {[n]} \ell} {d_n^{-p/2}} \sum_{\bd e\in \calW_p} \weight_{\bd Z(U)}(\bd w,\bd e)\,.\label{eq:traceExpansion}
    \end{equation}
    By the first observation above, only even Kikuchi walks contribute to~\eqref{eq:traceExpansion}. Hence,
    when $p$ is odd,
    \[
        \E \tr \left(Z_{w_1} \cdots Z_{w_p}\right) = 0 = \tau(\Sigma_{w_1}\cdots \Sigma_{w_p})\,,
    \]
    where the second equality follows
    from Lemma~\ref{lem:momentsGamma}.
    Suppose from now on that $p$ is even.
    
    An even Kikuchi walk of defect $\Delta$
    introduces $rp/2-\Delta$ vertices,
    so there are at most $O_{p,\ell,r}(n^{rp/2-\Delta})$ such walks. Since
    $|\weight_{\bd Z(U)}(\bd w, \bd e)|\le p^{\Delta/r}$ by the second observation and $d_n^{p/2} = \Omega_{p,\ell,r}(n^{rp/2})$, the total contribution of positive-defect walks converges to $0$, uniformly in $U$. 
    By the third observation, defect-$0$
    walks satisfy
    $\weight_{\bd Z(U)}(\bd w,\bd e) = \one\{\pi(\bd e)\in \calP(\bd w)\}$. Therefore,
    \begin{align*}
        \E \tr \left(Z_{w_1}\cdots Z_{w_p}\right) &= {d_n^{-p/2}} \sum_{\pi \in \calP(\bd w)} \left|\left\{\bd e\in \calW_p : \bd e\textnormal{ has defect $0$ and } \pi(\bd e) = \pi\right\}\right| + o(1)\\
        &= \frac {(n-\ell)\cdots (n-\ell-\frac {rp} 2+1)} {(r!)^{p/2} d_n^{p/2}} \sum_{\pi\in \calP(\bd w)} |\Hom(G_{\pi},K(\ell,r))| +o(1)\tag*{(Lemma~\ref{lem:countingGenusZero})}\\
        &\underset{n\to\infty}\longrightarrow {\binom \ell r}^{-p/2} \sum_{\pi\in \calP(\bd w)} |\Hom(G_{\pi},K(\ell,r))| = \tau\left(\Sigma_{w_1}\cdots \Sigma_{w_p}\right)\,.\tag*{(Lemma~\ref{lem:momentsGamma})}
    \end{align*}
    This proves convergence of the trace
    moments in expectation. It remains to show that
    \[
        \Var \tr\left(Z_{w_1}\cdots Z_{w_p}\right)\underset{n\to\infty}\longrightarrow 0\,.
    \]
    Expanding
    the expected square modulus by Wick's formula
    (as in~\eqref{eq:traceExpansion}), the
    terms corresponding to pairings with no
    edge between the two copies of the 
    length-$p$ sequence cancel with the squared modulus of the expectation. Every
    remaining pairing contains an edge between the two copies, meaning that
    any pair of Kikuchi walks 
    compatible with this
    pairing must share a vertex in the support.
    Such a pair of walks introduces at most
    $2\ell+rp-1$ vertices. 
    Hence, on the one hand the number
    of contributing pairs of walks is
    at most $O_{p,\ell,r}(n^{2\ell+rp-1})$.
    On the other hand, the normalization is
    \[{\binom n \ell}^{2} d_n^{p} = \Omega_{p,\ell,r}(n^{2\ell+rp})\,.\] Since every
    term has magnitude at most $1$ by the
    same argument as above, this implies that
    the variance is $O(1/n)$. Chebyshev's
    inequality therefore upgrades convergence
    in expectation to convergence in probability.
\end{proof}

\section{Norm Bound for Kikuchi Matrices}
\label{sec:norm}

In this section, we prove the following bound on the norm of the Kikuchi matrix, which is
tight up to a factor $O(\sqrt r)$.
Recall that the Kikuchi matrix $Y_n=Y_n(\ell,r)$ is defined in~\eqref{eq:KikuchiYSTconstruction}.

\begin{theorem}\label{thm:norm}
    There exist constants $c,C>0$ such that
    the following holds.
    For any integers $p,\ell,r,n\ge 1$ with $r\le \ell$,
    $n\ge \ell+r$, and $p\le c\min\{(n-\ell)/r^2,\sqrt{\ell(n-\ell)}/r\}$,
    \[
        \left(\E \tr Y_n^{2p}\right)^{\frac 1 {2p}}\le C \sqrt {\frac {\ell} r}\,.
    \]
    In particular, $\|Y_n\| \le C\sqrt{\ell/r}$ with high probability when $\ell\le c n/(r^2 \log^2 n)$.
\end{theorem}
Among all generalized Kikuchi matrices (Definition~\ref{def:generalizedKikuchi}),
the even trace moments are maximized for the
undirected Kikuchi matrices, so Theorem~\ref{thm:norm} directly generalizes
to arbitrary generalized Kikuchi matrices
(and in particular to the directed Kikuchi
matrices $X_n$).

The proof of Theorem~\ref{thm:norm} reduces to counting the number of objects appearing on the
right-hand side of~\eqref{eq:exactExpansionKikuchi}. The
crude counting of higher-defect walks
in Section~\ref{sec:weak} no longer
suffices when $p$ is growing with $n$.

Fix a tuple $(\bd e,\pi)$
such that $\bd e\in \calW_{2p}$ and $\pi\in \calP(\supp(\bd e))$. Define
\begin{equation}
    O_t \defeq \{i<t : \exists j\ge t,\; ij\in E(\pi)\}\label{eq:defOpened}
\end{equation}
to be the set of opening indices whose pair
is open immediately before index $t$.
Call a closing index $j$ with $ij\in E(\pi)$ {\em good} if $T_j\cap
    \bigcup_{k\in O_j\setminus\{i\}}\supp(e_k)=\varnothing$, and {\em bad} otherwise.
We call a pair $(i,u)$ where $i$ is an opening index and  $u\in H_i\cap (\bigcup_{j<i} U_j)$
a {\em coincidence} of the walk.
For a vertex $u$, we denote by $\delta(u)$ the number of times $u$ appears as
a coincidence. Note that $\sum_u \delta(u)=\Delta$.

\begin{lemma}\label{lem:boundBadVertices}
    If $\bd e$ has defect $\Delta$,
    then at most $\Delta$
    of the closing indices are bad.
\end{lemma}

\begin{proof}
    For a vertex $u$, the number of opening heads containing $u$ equals the number of
    closing tails containing $u$.
    When $j$ is a closing index, let $B_j \defeq T_j\cap
    \bigcup_{k\in O_j\setminus\{i\}}\supp(e_k)$, where $i$ is the matching opening of $j$.
    If $u\in U_0$, then every
    occurrence of $u$ in an opening head is
    a coincidence, so $u$ belongs to at most $\delta(u)$ sets $B_j$. If $u\notin U_0$, then the number of opening heads containing $u$
    is $\delta(u)+1$. However, the
    last occurrence of $u$ is a
    closing tail and no other
    open pair at that point
    can contain $u$. Therefore,
    in each case we have $|\{j : u\in B_j\}|\le \delta(u)$. Summing over $u$ concludes the proof.
\end{proof}

We say that $i\in O_t$ is {\em admissible}
at time $t$ if $\left|\supp(e_i)\cap U_{t-1}\right|=r$. We say furthermore that it
is {\em good} if $\supp(e_i)\cap U_{t-1}\cap \supp(e_j)=\varnothing$ for all $j\in O_t\setminus \{i\}$. When constructing a
paired Kikuchi walk from left to right,
admissible openings are the candidates that
can be closed at time $t$.
Good admissible openings create a good closing in this way.

\begin{lemma}\label{lem:goodClosing}
    The number of good admissible openings at any time is at most $\frac \ell r$.
\end{lemma}

\begin{proof}
    The candidate tails from the open pairs form a collection
    of pairwise disjoint subsets of $U_{t-1}$ of size
    $r$. Such a collection has size at most $|U_{t-1}|/r=\ell/r$.
\end{proof}

\begin{lemma}\label{lem:boundDefault}
    The number of admissible openings at any time is at most $\frac \ell r\left(1+\frac{2\Delta}\ell\right)$.
\end{lemma}

\begin{proof}
    Let $\Phi_t \defeq \sum_{i\in O_t} \left|\supp(e_i)\cap U_{t-1}\right|$.
    On the one hand, an opening in $O_t$ can be closed
    at step $t$ only if its support intersects $U_{t-1}$ in $r$ elements, so the number of admissible choices is at most $\Phi_t/r$. 
    On the other hand, we claim that for $u\in U_{t-1}$, we have $|\{i\in O_t : u\in \supp(e_i)\}|\le 2\delta(u) + \one\{u\notin U_0\}$. Summing this inequality over $u$ yields
    $\Phi_t \le 2 \sum_{u\in U_{t-1}} \delta(u) + |U_{t-1}\setminus U_0|\le 2\Delta + \ell$, as desired.

    It remains to prove the above inequality. Let $a,b,c,d$ denote the number of occurrences of $u$ in an opening head, closing head, opening tail, and closing tail before time $t$.
    Since $u\in U_{t-1}$, we have
    $a + b - c - d = \one \{u\notin U_0\}$.
    Further, $a\le \delta(u) + \one\{u\notin U_0\}$. Hence, the number of
    openings in $O_t$ whose support contains $u$ is
    $a+c - b-d = 2(a - d) - \one \{u\notin U_0\}\le 2\delta(u)  + \one\{u\notin U_0\}$.
\end{proof}

\begin{proof}[Proof of Theorem~\ref{thm:norm}]
    We start by fixing globally the defect of the walk $\Delta\ge 0$; the 
    sequence $(\delta_1, \ldots, \delta_{p})$, where
    $\delta_i$ is the number of coincidences at the $i$-th opening index (at most
    $2^{p+\Delta}$ choices since $\sum_i \delta_i = \Delta$); and the positions of the opening indices and bad closings (at most $3^{2p}$ choices).
    Afterwards, we construct the pair $(\bd e,\pi)$ from
    left to right for $t=1,\ldots,2p$. Assume that we have chosen $\supp(e_1),\ldots,\supp(e_{t-1})$ and $\{ij\in E(\pi) : 1\le i<j< t\}$. Then there are two cases.
    \begin{enumerate}
        \item $t$ is an opening index:
        The tail vertices can be identified as a subset of $r$ elements from $U_{t-1}$.
        The coincidence vertices in the head are
        a subset of the
        $\le 2pr$ previous vertices that have 
        previously appeared in a tail in the walk.
        The non-coincidence vertices in the head are a subset
        of the $n-\ell$ vertices that have not yet appeared in the walk. 
    
        \item $t$ is a closing index: It suffices to match $t$ to an element of $O_t$. If
        $t$ is a good (resp. bad) closing, there are
        $\frac \ell r$ (resp. $\frac{\ell} r \left(1+\frac{2\Delta} \ell\right)$) admissible choices by Lemma~\ref{lem:goodClosing} (resp. Lemma~\ref{lem:boundDefault}).
    \end{enumerate}
    Overall, this counting applied to the right-hand side of~\eqref{eq:exactExpansionKikuchi} gives
    \[
        \E \tr Y_n^{2p}\le \left(\frac 9 {d_n}\right)^p \sum_{\Delta\ge 0} \left[\left(\frac \ell r\right)^{p} \left(1 + \frac {2\Delta} \ell\right)^{\min\{p,\Delta\}} \sum_{\delta_1+\ldots+\delta_p=\Delta} \prod_{j=1}^p \binom \ell r \binom {2pr} {\delta_j} \binom {n-\ell}{r-\delta_j}\right]\,.
    \]
    The contribution of each opening index normalized by $d_n$ can be upper bounded by
    \[
        \frac 1 {d_n} \binom \ell r \binom {2pr} {\delta_j} \binom {n-\ell} {r-\delta_j} = \binom {2pr} {\delta_j} \frac {\binom {n-\ell} {r-\delta_j}} {\binom {n-\ell} r} \le \left(\frac {2pr^2} {n-\ell-r+1}\right)^{\delta_j}\,.
    \]
    The total contribution of the closing indices
    can be upper bounded by 
    \[
        \left(\frac \ell r\right)^{p} \left(1 + \frac {2\Delta} \ell\right)^{\min\{p,\Delta\}}\le \left(\frac {\ell} r\right)^p\left(1 + \frac {2p} \ell\right)^{\Delta}\,.
    \]
    Putting everything together,
    \[
        \E \tr Y_n^{2p}\le \left(\frac {18\ell} r\right)^p \sum_{\Delta\ge 0} \left(\frac {4pr^2} {n-\ell-r+1}\right)^\Delta \left(1 + \frac {2p} \ell\right)^{\Delta}\le \left(\frac {C\ell} r\right)^p\,,
    \]
    where the last inequality uses the fact that the series in $\Delta$ converges when $p$ satisfies the assumption of Theorem~\ref{thm:norm}. The concentration
    statement then follows from Markov's inequality.
\end{proof}
\section{Norm Convergence for Directed Kikuchi Matrices}
\label{sec:strongConvergence}

In this section, we establish Theorem~\ref{thm:directedIntro}.
In the previous section, we expanded
the trace moments of the Kikuchi matrix $Y_n$ in~\eqref{eq:exactExpansionKikuchi}.
The goal was to count the number
of paired Kikuchi walks $(\bd e,\pi)$, where $\pi$ satisfies $\supp(e_i)=\supp(e_j)$ for every $ij\in E(\pi)$. 
For the directed
Kikuchi matrix $X_n$, Wick's formula
imposes a stronger ``mirroring'' property (Lemma~\ref{lem:traceMoments}).
A {\em paired directed Kikuchi walk} of length $p$ is a tuple
$(\bd e,\pi)$ such that $\bd e\in \calW_{p}$ and
 $e_i=e_j^\ast$ for all $ij\in E(\pi)$. Then by~\eqref{eq:exactExpansionDirected},
 the $p$-th moment of $X_n$ matches (up to an explicit normalization) the number of paired directed
Kikuchi walks of length $p$.

Recall that the proof of weak convergence 
relies on the fact that defect-$0$ walks 
have transports of 
excess size $0$ and cost $0$ (Lemma~\ref{lem:transportZero}). 
We now prove a generalization of this result.

\begin{lemma}\label{lem:existenceTransport}
    Any defect-$\Delta$ paired directed Kikuchi walk has a Kikuchi transport of excess size $\Delta$ and cost at most $3\Delta$.
\end{lemma}

\begin{proof}
    Let $(\bd e,\pi)$ be a defect-$\Delta$ paired directed Kikuchi walk. Let $\rho_0(u)\defeq u$. For each $t=1,\ldots,2p-1$,
    \begin{enumerate}
        \item Suppose that $t+1$ is an opening index with closing $j$. For any $u\in T_{t+1}$ such that $\rho_{t-1}(\phi_t(u))\in \rho_{i-1}(T_i)$ for some opening $i\le t$ whose pair crosses $(t+1, j)$, let $\rho_t(u)$
        be a fresh new label. For every other
        $u\in U_t$, let $\rho_t(u) \defeq \rho_{t-1}(\phi_t(u))$.
        
        \item Suppose that $t+1$ is a closing index with opening $i$. Then define
        \[
            \rho_{t}(u)\defeq \begin{cases}\rho_{i-1}(\phi_i(u)) & \text{if $u\in T_{t+1}$,}\\
           \rho_{t-1}(\phi_t(u)) & \text{if $u\in U_{t}\setminus T_{t+1}$.}
            \end{cases}
        \]
    \end{enumerate}
    Finally, we set $\rho_{2p}(u) \defeq \rho_{2p-1}(\phi_{2p}(u))$.
    We show that this algorithm constructs a 
    Kikuchi transport $\bd \rho = (\rho_0, \ldots, \rho_{2p})$ of excess size at most $\Delta$
    and cost at most $3\Delta$.
    By construction, Steps 1 and 2 ensure that both compatibility conditions of Definition~\ref{def:transport} are satisfied:
    \[
        \rho_{i-1}(T_i)=\rho_{j-1}(T_j)\quad \forall ij\in E(\pi)\quad\text{ and }\quad \rho_{i-1}(T_i)\cap \rho_{i'-1}(T_{i'})=\varnothing\quad \forall \{ij,i'j'\}\in E(G_\pi)\,.
    \]
    For any label $a$,
    consider
    \[
        O_t(a) \defeq \left\{i\le t : a\in \rho_{i-1}(T_i)\text{ and }\exists j>t, \;ij\in E(\pi)\right\}\,.
    \]
    The crossing compatibility condition
    directly implies:
    \begin{claim}\label{claim:orderingOT}
        For any $t$ and $a$, out of all the
        openings in $O_t(a)$, the one that
        closes first is $\max O_t(a)$.
    \end{claim}
    
    \begin{claim}\label{claim:structuralAlg}
        For any $t\ge 1$, the map $\rho_t$ is
        injective. Furthermore, if $u\in U_t$
        is such that $a\defeq \rho_{t-1}(\phi_t(u))$ satisfies $O_t(a)\neq \varnothing$, then $u\in H_i$ and $a = \rho_{i-1}(\phi_i(u))$, where $i\defeq \max O_t(a)$.
    \end{claim}

    \begin{proof}[Proof of Claim~\ref{claim:structuralAlg}]
        We proceed by induction on $t$. The statements are immediate for $t=0$, as $\rho_0=\textnormal{id}$ and $O_0(a)=\varnothing$ for all labels $a$.
        Now suppose that the claims hold for $0, \ldots, t-1$. We first prove the
        ``Furthermore'' part. Let $u\in U_t$, $a=\rho_{t-1}(\phi_t(u))$, and suppose that $O_t(a)\neq \varnothing$. There
        are two cases:
        \begin{itemize}
            \item Suppose that $t$ is an opening. We
            either have $O_t(a)=O_{t-1}(a)$
            if $a\notin \rho_{t-1}(T_t)$,
            or $O_t(a) = O_{t-1}(a)\cup \{t\}$
            if $a\in \rho_{t-1}(T_t)$.
            By injectivity of $\rho_{t-1}$, these two cases are equivalent respectively to $u\notin H_t$ and $u\in H_t$.
            In the former case, $a = \rho_{t-1}(u) = \rho_{t-2}(\phi_{t-1}(u))$, so the claim follows by induction.
            In the latter case, $t=\max O_t(a)$ and the claim is immediate.
            
            \item Suppose that $t$ is a closing, with
            opening $k$. Assume first $u\notin H_t$. By the same
            reasoning as in the opening case, $a=\rho_{t-2}(\phi_{t-1}(u))$.
            We claim that $k\notin O_{t-1}(a)$: otherwise, Claim~\ref{claim:orderingOT} shows that $k=\max O_{t-1}(a)$, which by induction would imply that $u\in H_k =T_t$, contradicting $u\in U_t$. Hence, $O_t(a)=O_{t-1}(a)$ and the claim follows by induction. Next, assume $u\in H_t$. Then by the
            construction in Step 2, $a = \rho_{k-1}(\phi_k(\phi_t(u))) = \rho_{k-1}(u)$ (using $\phi_k = \phi_t^{-1}$). By crossing
            compatibility, we must have $O_{k-1}(a)=O_t(a)$. In particular, $O_{k-1}(a)\neq\varnothing$ and therefore $a=\rho_{k-2}(\phi_{k-1}(u))$; the claim then
            follows from the induction hypothesis at step $k-1$.
        \end{itemize}
        The construction of $\rho_t$
        when $t+1$ is an opening clearly inherits injectivity from $\rho_{t-1}$. Suppose that $t+1$ is a
        closing index with opening $i$.
        The labels assigned by Step 2 to $T_{t+1}$ are pairwise distinct, so it
        remains to show that none of these
        labels can occur outside of $T_{t+1}$.
        For any $a\in \rho_{i-1}(T_{i})$, we have $i\in O_t(a)$,
        so $i=\max O_t(a)$ by Claim~\ref{claim:orderingOT} and
        so any $u\in U_{t}$ such that
        $\rho_{t-1}(\phi_t(u))\in \rho_{i-1}(T_i)$
        must satisfy $u\in H_i=T_{t+1}$ by 
        induction. This concludes the induction
        step.
    \end{proof}

    For the transport size, charge every replacement of a label $a=\rho_{t-1}(\phi_t(u))$ in Step $1$ to
    $(u,\max O_t(a))$. By
    Claim~\ref{claim:orderingOT}, the
    pair $ij\in E(\pi)$ containing $i\defeq \max O_t(a)$ is the first pair in $O_t(a)$ to close. 
    After its replacement, the label $a$
    disappears from the walk, and its first
    possible reappearance is at the closing
    of an opening in $O_t(a)$. Hence, $a$
    cannot reappear before step $j$, while after $j$, the pair $ij$ is no longer open.
    In summary, every pair $(u,i)$ can be
    charged at most once, and the pairs that are charged must satisfy that $i$ is an opening index and $u\in H_i$ (by Claim~\ref{claim:structuralAlg}). When
    $u\notin U_0$, the last occurrence of
    $u$ in the walk is a closing, and its
    matching opening index cannot be
    charged. Thus, each $u$ can be
    charged at most $\delta(u)$ times,
    and the total number of replacements in
    Step 1 is at most $\sum_u \delta(u)=\Delta$.

    For the transport cost, by the
    above argument, the number
    of switches incurred by Step 1 is
    at most $\Delta$. Let $B(u)$ be the
    number of switches of $u$ incurred
    during Step 2. At every such switch occurring at a closing index $t+1$ with matching opening $i$, we have $u\in T_{t+1} = H_i$, so $B(u)\le \delta(u)+1$. Next, we claim
    that if $B(u)>0$, then $u$ must incur at
    least one switch in Step 1. 
    Assuming this claim, the
    total number of switches due to Step 2 is at most $\sum_{u : B(u)>0} (\delta(u) + 1)\le 2\Delta$, as desired.

    To prove the remaining claim, let
    $t$ be the first iteration
    for which Step 2 incurs a switch at $u$.
    Let $t+1$ have opening index $i$, and let
    $a\defeq \rho_{i-1}(\phi_i(u))$ be the new
    label. Consider the last iteration
    $i\le s<t$ for which $a\in \rho_{s-1}(U_{s-1})\setminus \rho_s(U_s)$, and let
    $v\in U_s$ be defined by $a = \rho_{s-1}(\phi_s(v))$. Then $\max O_s(a) = i$ (otherwise, the maximizing
    pair would close before the pair of $i$ and
    contradict the maximality of $s$),
    so by Claim~\ref{claim:structuralAlg},
    $a=\rho_{i-1}(\phi_i(v))$. By injectivity, this implies $v=u$. By minimality of $t$, the label
    of $u$ at time $s$ could not have been replaced at
    a Step 2, so it must be at a Step 1.
\end{proof}

\begin{lemma}\label{lem:boundSlot}
    For any fixed $D_1, \ldots, D_{2p}\in \binom {[\ell+\Delta]} \ell$,
    \[
        \sum_{\pi\in \calP(2p)}  |\{\varphi \in \Hom(G_\pi, K(\ell+\Delta,r)) : \varphi(ij)\subseteq D_i \cap D_j\;\forall ij\in E(\pi)\}|\le {\binom \ell r}^p \tau(\Sigma^{2p})\,.
    \]
\end{lemma}

\begin{proof}
    Let $\{S_a : a \in \binom {[\ell+\Delta]} r\}$ be a $K(\ell+\Delta,r)$-independent semicircular system, and
    \[
        Q_i \defeq \sum_{a\in \binom {D_i} r} S_a\quad \forall i\in [2p]\,.
    \]
    Each $Q_i$ is a sum of a $K(\ell,r)$-independent semicircular system, so $\tau(Q_i^{2p}) = {\binom \ell r}^p\tau(\Sigma^{2p})$.
    Further, by Lemma~\ref{lem:momentsGamma}, the
    left-hand side of Lemma~\ref{lem:boundSlot} is precisely $\tau(Q_1\cdots Q_{2p})$. By Hölder's inequality (Lemma~\ref{lem:holder}), we obtain
    \[
        \tau(Q_1\cdots Q_{2p})\le \prod_{i=1}^{2p} \tau(Q_i^{2p})^{\frac 1 {2p}} = {\binom \ell r}^{p} \tau(\Sigma^{2p})\,,
    \]
    which concludes the proof.
\end{proof}

\begin{proof}[Proof of Theorem~\ref{thm:directedIntro}]
    We count the
    triples $(\bd e,\pi,\bd \rho)$ satisfying the conclusion of Lemma~\ref{lem:existenceTransport}.
    By existence of the labeling, this upper bounds the number of paired directed Kikuchi walks (that the trace moments count exactly). 
    
    Fix
    the defect $\Delta\in \{0,\ldots,rp\}$ and a set of switches $\sigma\subseteq [2p]\times [\ell+\Delta]^2$ from Definition~\ref{def:switches}, encoded as triples $(t,\rho_{t-1}(\phi_t(u)),\rho_t(u))$, of size $|\sigma|\le 3\Delta$. There are at
    most $(2p(\ell+\Delta)^2+1)^{3\Delta}$ choices for
    this set.
    By construction, the switches determine the labels of the states $(\rho_{i}(U_i))_{i\in [2p]}$.
    Given that information, Lemma~\ref{lem:boundSlot} applied to $D_i = \rho_{i-1}(U_{i-1})$ then
    implies that the total number of
    remaining possibilities for the pairing
    $\pi\in\calP(2p)$ and the
    sequence $(\rho_{i-1}(T_i))_{i\in [2p]}$ is at most
    ${\binom \ell r}^p \tau(\Sigma^{2p})$. Finally, fix the number of coincidences $\delta_j$ at the $j$-th opening index.

    We construct the maps $\rho_i$ and the hyperedges $e_i$ for $i=1,\ldots,2p$ from left to right. At the $j$-th opening index $i\in [2p]$,
    the tail indices $T_i$ can be
    obtained from $\rho_{i-1}(T_i)$
    using injectivity of $\rho_{i-1}$.
    For the head $H_i$, we must choose $\delta_j$ vertices among
    the $\le 2pr$ vertices that have
    appeared previously in a tail,
    and $r-\delta_j$ vertices among
    the $n-\ell$ vertices outside of the current state. At
    a closing index, the head and
    tail are uniquely determined
    from the matching opening index,
    using the mirroring property.\footnote{This is where we use the fact that the Kikuchi matrix is directed.} The map $\rho_i$ is then obtained from $\rho_{i-1}$ using the set of switches.
    Overall,
    \begin{align*}
        \E \tr X_n^{2p}&\le \frac {{\binom \ell r}^p} {d_n^p  } \tau(\Sigma^{2p}) \sum_{\Delta=0}^{rp}  (2p(\ell+\Delta)^2+1)^{3\Delta} \sum_{\delta_1+\ldots+\delta_p=\Delta} \prod_{j=1}^p \binom {2pr} {\delta_j} \binom {n-\ell} {r-\delta_j}\\
        &= \tau(\Sigma^{2p}) \sum_{\Delta=0}^{rp} (2p(\ell+\Delta)^2+1)^{3\Delta} \sum_{\delta_1+\ldots+\delta_p=\Delta} \prod_{j=1}^p \frac {\binom {2pr} {\delta_j} \binom {n-\ell} {r-\delta_j}} {\binom {n-\ell} r}
    \end{align*}
    The combinatorial coefficient can be bounded by
    \[
        \sum_{\delta_1+\ldots+\delta_p=\Delta} \prod_{j=1}^p \frac {\binom {2pr} {\delta_j} \binom {n-\ell} {r-\delta_j}} {\binom {n-\ell} r}\le  \sum_{\delta_1+\ldots+\delta_p=\Delta} \prod_{j=1}^p \left(\frac {4p r^2} {n-\ell}\right)^{\delta_j}\le \left(\frac {4p^2 r^2} {n-\ell}\right)^{\Delta}\,.
    \]
    Every term
    in the sum under the power $\Delta$ is then bounded by
    \[
        (2p(\ell+\Delta)^2+1)^{3} \cdot \left(\frac {4p^2 r^2} {n-\ell}\right)\lesssim \frac {p^5 \max\{\ell,rp\}^{6} r^2}{n-\ell} \underset{n\to\infty}\longrightarrow 0\,,
    \]
    provided, for example, that $r\le C$ and $\max\{p,\ell\}\le n^{c}$ for constants $c,C>0$.
    The desired upper bound then follows from Markov's inequality, and the matching lower bound follows from weak convergence.
\end{proof}
\section{Convergence to the \texorpdfstring{$q$}{q}-Gaussian system}
\label{sec:qGaussian}

In this section, we prove Theorem~\ref{thm:qGaussianConvergence}. 
For a graph $\Gamma$ on $N$ vertices, we denote by $\lambda_1(\Gamma)\ge \ldots \ge \lambda_N(\Gamma)$ the eigenvalues of its adjacency matrix.
We 
start by computing the spectrum of the adjacency matrix of the Kneser graphs.

\begin{lemma}[See {\cite[Theorem 9.4.3]{godsilCombinatorics}}]\label{lem:eigenvaluesKneser}
    For $\ell>2r$, the adjacency matrix of
    $K(\ell,r)$ has eigenvalues
    \[
        (-1)^i \binom {\ell-r-i} {r-i}\qquad i = 0,\ldots,r\,,
    \]
    and the eigenvalue for $i=0$ has multiplicity $1$.    
\end{lemma}

\noindent We will use Lemma~\ref{lem:eigenvaluesKneser}
to show that the Kneser graphs are quasirandom
in the sense of Chung-Graham-Wilson~\cite{chungGramWilson}. One equivalent
condition for quasirandomness is the convergence
of all homomorphism counts to those of an \Erdos-\Renyi random graph with matching density.

\begin{lemma}[See~\cite{chungGramWilson}]\label{lem:quasiRandomEquivalence}
    Let $(\Gamma_\ell)_{\ell\ge 1}$ be a
    sequence of regular graphs whose edge density
    converges to $q\in [0,1]$. Suppose that
    \[
        \max_{i\ge 2} |\lambda_i(\Gamma_\ell)| = o(|V(\Gamma_\ell)|)\,.
    \]
    Then for every fixed graph $G$,
    \[
        \frac {|\Hom(G, \Gamma_\ell)|}{|V(\Gamma_\ell)|^{|V(G)|}}\underset{\ell\to\infty}\longrightarrow q^{|E(G)|} \,.
    \]
\end{lemma}

We use the following convention in this section: denoting by $K_2$ the
graph consisting of a single edge, the edge density of $\Gamma$ is $ {|\Hom(K_2,\Gamma)|}/{|V(\Gamma)|^2}$.

\begin{lemma}[{See~\cite[Section 3]{qGaussianSupport}}]\label{lem:supportQGaussian}
    Let $S$ be a $q$-Gaussian variable (Definition~\ref{def:qGaussianSystem}) for
    $q\in [0,1)$. Then there exists a probability distribution $\mu_q$
    with full support on
    \[
        \left[-\frac 2 {\sqrt{1-q}},\frac 2 {\sqrt{1-q}}\right]
    \]
    such that $\tau(S^p) = \int x^p\diff \mu_q(x)$ for all $p\ge 1$.
    In particular, we have $\|S\| = \frac 2 {\sqrt{1-q}}$.
\end{lemma}

\subsection{Weak convergence}

\begin{lemma}\label{lem:qGaussianWeak}
    Let $\alpha> 0$ and $\bd \Sigma$ be as in Theorem~\ref{thm:weakIntro}. Suppose that $r=r(\ell)$ is such that $r^2/\ell \to \alpha$ as $\ell\to\infty$. Then $\bd \Sigma$ converges weakly as $\ell\to\infty$
    to the $e^{-\alpha}$-Gaussian system with $k$
    variables.
\end{lemma}

\begin{proof}
    We compute the limit of every 
    monomial
    indexed by $\bd w\in [k]^p$.
    The moment formula for a $\Gamma$-independent semicircular system (Lemma~\ref{lem:momentsGamma}) gives
    \begin{equation*}
        \tau\left(\Sigma_{w_1}\cdots \Sigma_{w_p}\right) = {{\binom \ell r}^{-p/2}} \sum_{\pi \in \calP(\bd w)} |\Hom(G_\pi,K(\ell,r))|\,.
    \end{equation*}
    The Kneser graph $K(\ell,r)$ is regular, with edge density
    \[
        \frac {\binom {\ell-r} r} {\binom \ell r} = \frac {((\ell-r)!)^2} {(\ell-2r)! \ell!} \underset{\ell\to\infty} {\longrightarrow} e^{-\alpha}\,,
    \]
    which can be verified by applying Stirling's formula. Moreover, Lemma~\ref{lem:eigenvaluesKneser} gives
    \begin{equation*}
        \frac 1 {\binom \ell r}\max_{i\ge 2} |\lambda_i(K(\ell,r))| \underset{\ell\to\infty} {\longrightarrow} 0\,.
    \end{equation*}
    Therefore, by Lemma~\ref{lem:quasiRandomEquivalence}, all
    homomorphism densities converge to those
    of an \Erdos-\Renyi random graph, i.e.,
    \[
        \tau\left(\Sigma_{w_1}\cdots \Sigma_{w_p}\right) \underset{\ell\to\infty}\longrightarrow \sum_{\pi \in \calP(\bd w)} e^{-\alpha|E(G_\pi)|}\,.
    \]
    This is exactly the moment formula for the
    $e^{-\alpha}$-Gaussian system (Lemma~\ref{lem:momentsQGaussian}).
\end{proof}

\subsection{Norm convergence for quasirandom independence graphs}

In this section, given a graph $\Gamma$,
we let $(L_u)_{u\in V(\Gamma)}$ be the
creation operators associated to the
$\Gamma$-independent semicircular system,
and define
\[
    \Sigma(\Gamma) \defeq \frac 1 {\sqrt{|V(\Gamma)|}} \sum_{u\in V(\Gamma)} (L_u + L_u^\ast)
\]
to be its normalized sum. We denote by $\omega(\Gamma)$ the maximum clique size of a
graph $\Gamma$.
The main result of this
subsection is:

\begin{proposition}\label{prop:pseudoRandom}
    Let $(\Gamma_{\ell})_{\ell\ge 1}$
    be a sequence of regular graphs whose edge
    density converges to $q\in [0,1)$.
    Suppose that
    \begin{equation}
        \omega(\Gamma_\ell)\max\{0,\lambda_2(\Gamma_\ell)\} + |\lambda_{|V(\Gamma_\ell)|}(\Gamma_\ell)|= o(|V(\Gamma_\ell)|)\quad\text{as $\ell\to\infty$}\,.\label{eq:pseudoRandomCondition}
    \end{equation}
    Then 
    \[
        \|\Sigma(\Gamma_\ell)\| \underset{\ell\to\infty}\longrightarrow \frac 2 {\sqrt {1-q}}\,.
    \]
\end{proposition}

Recall that by Lemma~\ref{lem:supportQGaussian}, the limit is precisely
the edge of the support of the $q$-Gaussian distribution. 
The proof of Proposition~\ref{prop:pseudoRandom} divides
into a lower bound, which follows from weak
convergence, and an upper bound.
For the latter part, existing 
Khintchine inequalities for $\Gamma$-independent semicircular systems~\cite{collinsEigenvalueBound,youssefTuran} are not sufficiently
sharp.

Instead, we exploit the fact that
our graphs have small second eigenvalue
and use the following bound:

\begin{proposition}[\cite{BKNPW2026-Paper2}]\label{prop:pseudoRandomUpperBound}
    For any graph $\Gamma$ and $q\in [0,1)$,
    \[
        \|\Sigma(\Gamma)\|\le \frac 2 {\sqrt{1-q}} \sqrt{1 + \frac {\omega(\Gamma) \max\{0,\lambda_1(\Gamma-qJ)\}} {|V(\Gamma)|}}\,,
    \]
    where $J$ denotes the all-ones matrix.
\end{proposition}

\noindent The proof of Proposition~\ref{prop:pseudoRandomUpperBound} can be found in the parallel work~\cite[Corollary 5.3]{BKNPW2026-Paper2}.

\begin{proof}[Proof of Proposition~\ref{prop:pseudoRandom} from Proposition~\ref{prop:pseudoRandomUpperBound}]
    The same argument as in Lemma~\ref{lem:qGaussianWeak} shows weak
    convergence: by Lemma~\ref{lem:momentsGamma},
    \[
        \tau(\Sigma(\Gamma_\ell)^p) = \frac 1 {|V(\Gamma_\ell)|^{p/2}} \sum_{\pi \in \calP(p)} |\Hom(G_\pi,\Gamma_\ell)|\,.
    \]
    The assumption~\eqref{eq:pseudoRandomCondition}
    is stronger than that of Lemma~\ref{lem:quasiRandomEquivalence}, so
    we obtain
    \[
        \tau(\Sigma(\Gamma_\ell)^p) \underset{\ell\to\infty}\longrightarrow \sum_{\pi \in \calP(p)} q^{|E(G_\pi)|} = \tau(S^p)\,,
    \]
    where $S$ is a single $q$-Gaussian noncommutative random variable
    defined through Theorem~\ref{thm:existenceToeplitz} (see Lemma~\ref{lem:momentsQGaussian} for the 
    associated moment formula). Weak convergence directly implies a lower bound on the norm. Indeed, by~\eqref{eq:faithfulAutomatic},
    for any fixed $p\ge 1$,
    \[
        \|\Sigma(\Gamma_\ell)\|\ge \tau \left(\Sigma(\Gamma_\ell)^{2p}\right)^{\frac 1 {2p}}\underset{\ell\to\infty}\longrightarrow \tau(S^{2p})^{\frac 1{2p}}\,.
    \]
    By Lemmas~\ref{lem:faithful} and~\ref{lem:supportQGaussian}, taking the $p\to\infty$ limit on the right-hand side yields
    \[
        \liminf_{\ell\to\infty} \|\Sigma(\Gamma_\ell)\|\ge \frac 2 {\sqrt{1-q}}\,.
    \]
    To prove the matching upper bound, 
    we use Proposition~\ref{prop:pseudoRandomUpperBound}. Let $q_\ell$ be the edge density of $\Gamma_\ell$, so that $q_\ell$ converges to $q$. Since $\Gamma_\ell$ is regular,
    we have $\max\{0,\lambda_2(\Gamma_\ell)\}=\lambda_1(\Gamma_\ell-q_\ell J)$. Combining Proposition~\ref{prop:pseudoRandomUpperBound} with the assumption~\eqref{eq:pseudoRandomCondition},
    we obtain
    \begin{equation*}
        \limsup_{\ell\to\infty} \|\Sigma(\Gamma_\ell)\|\le \frac 2 {\sqrt {1-q}}\,.
    \end{equation*}
    This concludes the proof.
\end{proof}

\subsection{Applications to the Kneser graph}

The graph $K(\ell,r)$ has $N_\ell \defeq \binom \ell r$ vertices, and a clique is a collection of disjoint sets of $[\ell]$ of size $r$, so its maximum clique size is $\lfloor \ell/r\rfloor$. 
Combining this with the eigenvalues of the Kneser graph from Lemma~\ref{lem:eigenvaluesKneser},
\begin{align*}
    \frac {\omega(K(\ell,r))\max\{0,\lambda_2(K(\ell,r))\}} {N_\ell} &= \frac {\lfloor \ell/r\rfloor \binom {\ell-r-2} {r-2}} {\binom \ell r}\le \frac {r-1} {\ell-1}\,,\\
    \frac {|\lambda_{N_\ell}(K(\ell,r))|} {N_\ell} &= \frac {\binom {\ell-r-1} {r-1}} {\binom \ell r} \le \frac r \ell\,.
\end{align*}

\begin{proposition}\label{prop:upperBoundKneser}
    $\|\Sigma(K(\ell,r))\| \asymp \max\left\{1,\frac{\sqrt \ell} r\right\}$
\end{proposition}

\begin{proof}
    The edge density of $K(\ell,r)$ is
    $q \defeq {\binom {\ell-r} r}/{\binom \ell r}$, from which one can verify that
    \[
        \frac 2 {\sqrt{1-q}} \asymp \max\left\{1,\frac {\sqrt \ell} r\right\}\,.
    \]
    Proposition~\ref{prop:pseudoRandomUpperBound} then implies the desired upper bound. For the lower bound, the normalization $\tau(\Sigma(K(\ell,r))^2)=1$ gives the constant term, and the inequality~\eqref{eq:lowerBoundNorm}
    gives the other one.
\end{proof} 

\begin{proof}[Proof of Theorem~\ref{thm:qGaussianConvergence}]
Lemma~\ref{lem:qGaussianWeak} combined with Theorem~\ref{thm:weakConvergence} directly implies the weak convergence assertion in Theorem~\ref{thm:qGaussianConvergence}.
Theorem~\ref{thm:directedIntro} shows that
$\|X_n\|\to \|\Sigma(K(\ell,r))\|$ in probability as $n\to\infty$. 
In the regime $r^2/\ell\to \alpha$, condition~\eqref{eq:pseudoRandomCondition} 
is verified and Proposition~\ref{prop:pseudoRandom} implies that 
\[
    \|\Sigma(K(\ell,r))\|\underset{\ell\to\infty}\longrightarrow \frac 2 {\sqrt{1-e^{-\alpha}}}\,,
\]
so $\|X_n\|$ also converges in probability
to that constant in the double limit $n\to\infty$, then $\ell\to\infty$.
We generalize this to any polynomial $P$.
Letting $\rho=2/\sqrt{1-e^{-\alpha}}$ and $S$ be a $e^{-\alpha}$-Gaussian variable,
Lemmas~\ref{lem:supportQGaussian} and~\ref{lem:faithful} give
\[
    \|P(S)\| = \lim_{p\to\infty} \tau((P(S)^\ast P(S))^p)^{\frac 1 {2p}} = \lim_{p\to\infty}  \left(\int |P(x)|^{2p} \diff \mu_{e^{-\alpha}}\right)^{\frac 1 {2p}} = \max_{|x|\le \rho} |P(x)|\,.
\]
Since $\|P(X_n)\|\le \max_{|x|\le \|X_n\|} |P(x)|$, we obtain for every
fixed $\eps>0$
\[
    \limsup_{\ell\to\infty} \limsup_{n\to\infty} \Pr\left(\|P(X_n)\| > \|P(S)\|+\eps\right) = 0\,.
\]
The matching lower bound follows from weak
convergence: for every fixed $p\ge 1$,
\[
    \|P(X_n)\|^{2p}\ge \tr \left((P(X_n)^\ast P(X_n))^p\right)\,,
\]
which converges in the double limit in probability to $\tau\left((P(S)^\ast P(S))^p\right)$.
Taking the further limit $p\to\infty$ 
and using Lemma~\ref{lem:faithful} concludes the proof of Theorem~\ref{thm:qGaussianConvergence}.
\end{proof}

\bibliographystyle{alpha}
{\small \bibliography{main}}

@inproceedings{bandeira2025randommatrices,
  author        = {Bandeira, Afonso S.},
  title         = {Random Matrices, Intrinsic Freeness, and Sharp
                   Non-Asymptotic Inequalities},
booktitle = {Proceedings of the International Congress of Mathematicians (ICM)},
  year = 2026,
  pages = {289--309}
}

@inproceedings{vanhandel2025strongconvergence,
  author        = {Ramon van Handel},
  title         = {Strong Convergence: A Short Survey},
  booktitle = {Proceedings of the International Congress of Mathematicians (ICM)},
  pages = {145--165},
  year = 2026
}

@article{bandeira2025randomstrasse101,
  author        = {Bandeira, Afonso S. and Kireeva, Anastasia and
                   Maillard, Antoine and R{\"o}dder, Almut},
  title         = {{Randomstrasse101}: Open Problems of 2024},
  journal = {arXiv:2504.20539},
  year          = {2025}
}

@phdthesis{Manohar-2024-KikuchiMatrixMethod,
  author={Manohar, Peter},
  title={New Spectral Techniques in Algorithms, Combinatorics, and Coding Theory: The {Kikuchi} Matrix Method},
  school={Carnegie Mellon University},
  year={2024}
}

@article {BBvH23,
    AUTHOR = {Bandeira, Afonso S. and Boedihardjo, March T. and van Handel,
              Ramon},
     TITLE = {Matrix concentration inequalities and free probability},
   JOURNAL = {Inventiones Mathematicae},
    VOLUME = {234},
      YEAR = {2023},
    NUMBER = {1},
     PAGES = {419--487}
}

@article{CGVTVH25,
      title={{A new approach to strong convergence}}, 
      author={Chen, Chi-Fang and {Garza-Vargas}, Jorge  and Tropp, Joel A. and van Handel, Ramon },
      journal = {Annals of Mathematics},
      year={2026},
      volume = 203,
      number = 2,
      pages = {555--602}
}

@article{BCSvH24,
      title={Matrix Concentration Inequalities and Free Probability {II}. {T}wo-sided Bounds and Applications}, 
      author={Afonso S. Bandeira and Giorgio Cipolloni and Dominik Schröder and Ramon van Handel},
      year={2026},
      journal = {Communications of the American Mathematical Society},
      volume = 6,
      pages = {896--946}
}

@book{NS06, 
    AUTHOR = {Nica, Alexandru and Speicher, Roland},
     TITLE = {Lectures on the combinatorics of free probability},
 PUBLISHER = {Cambridge University Press},
      YEAR = {2006}
}

@article {HT05,
    AUTHOR = {Haagerup, Uffe and Thorbj{\o}rnsen, Steen},
     TITLE = {{A new application of random matrices: {${\rm Ext}(C^*_{\rm red}(F_2))$} is not a group}},
   JOURNAL = {Annals of Mathematics},
    VOLUME = {162},
      YEAR = {2005},
    NUMBER = {2},
     PAGES = {711--775}
}

@article{CGVvH24,
      title={A New Approach to Strong Convergence {II}. {The} Classical Ensembles}, 
      author={{Chi-Fang} Chen and Jorge {Garza-Vargas} and Ramon van Handel},
      year={2026},
      journal = {Geometric and Functional Analysis},
      volume = 36,
      pages = {831--890}
}

@inproceedings{HsiehKothariMohanty2022simpleMoore,
  author    = {Jun-Ting Hsieh and
               Pravesh K. Kothari and
               Sidhanth Mohanty},
  title     = {A Simple and Sharper Proof of the Hypergraph {Moore} Bound},
  booktitle = {Symposium on
               Discrete Algorithms (SODA)},
  pages     = {2324--2344},
  year      = {2023}
}

@inproceedings{GuruswamiKothariManohar2022,
  author    = {Venkatesan Guruswami and
               Pravesh K. Kothari and
               Peter Manohar},
  title     = {Algorithms and Certificates for {Boolean} {CSP} Refutation:
               ``Smoothed is No Harder than Random''},
  booktitle = {Symposium on Theory of Computing (STOC)},
  pages     = {678--689},
  year      = {2022}
}

@article{BKNPW2026-Moore,
  title         = {The Hypergraph {M}oore Bound},
  author        = {Bandeira, Afonso S. and Kunisky, Dmitriy and {Nizi{\'c}-Nikolac}, Petar and Pesenti, Lucas and Wang, Robert},
  year          = {2026},
  journal          = {arXiv:2607.14068v2}
}

@unpublished{BKNPW2026-Paper2,
  title         = {Matrix Concentration and Equivalent Operators on {Fock} Spaces},
  author        = {Bandeira, Afonso S. and Kunisky, Dmitriy and {Nizi{\'c}-Nikolac}, Petar and Pesenti, Lucas and Wang, Robert},
  note = {arXiv},
  year          = {2026}
}

@inproceedings{kothari2025smoothtradeofftensorpca,
      title={Smooth Trade-off for Tensor {PCA} via Sharp Bounds for {K}ikuchi Matrices}, 
      author={Pravesh K. Kothari and Jeff Xu},
      year={2026},
      booktitle = {Symposium on Discrete Algorithms (SODA)},
      pages = {2617--2632}
}

@article{graphonCLT1,
    title = {Graphon-Theoretic Approach to Central Limit Theorems for $\epsilon$-Independence},
    author = {Cébron, Guillaume and {Oliveira Santos}, Patrick and Youssef, Pierre},
    journal = {arXiv:2411.13062},
    year = 2024
}

@article{proofGammaRepresentation,
    author = {Marek Bożejko and Roland Speicher},
    title = {{Completely positive maps on Coxeter groups, deformed commutation relations, and operator spaces}},
    volume = 300,
    pages = {97--120},
    year = 1994,
    journal = {Mathematische Annalen}
}

@article{collinsEigenvalueBound,
    author = {Benoît Collins and  Akihiro Miyagawa},
    title = {{Operator-valued Khintchine inequality for $\varepsilon$-free semicircles}},
    journal = {Bulletin of the London Mathematical Society},
    volume = 58,
    number = 1,
    pages = {e70229},
    year = 2026
}

@article{youssefTuran,
    title = {{Khintchine inequalities, trace monoids and Turán-type problems}},
    author = {Patrick {Oliveira Santos} and Raghavendra Tripathi and Pierre Youssef},
    journal = {arXiv:2506.02517},
    year = 2025
}

@article{gammaIndependenceDefinition,
    title = {{Mixtures of classical and free independence}},
    author = {Roland Speicher and Janusz Wysoczański},
    journal = {Archiv der Mathematik},
    year = 2016,
    volume = 107,
    pages = {445--453}
}

@inproceedings{WeinElAlaouiMoore2019,
  author    = {Alexander S. Wein and Ahmed {El Alaoui} and Cristopher Moore},
  title     = {The {Kikuchi} Hierarchy and Tensor {PCA}},
  booktitle = {Symposium on
               Foundations of Computer Science (FOCS)},
  pages     = {1446--1468},
  year      = {2019}
}

@article{NCK,
    title = {{Non commutative Khintchine and Paley inequalities}},
    author = {Françoise {Lust-Piquard} and Gilles Pisier},
    year = 1991,
    volume = 29,
    pages = {241--260},
    journal = {Arkiv för Matematik}
}

@article{charlesworthCollins,
    title = {{Matrix models for $\varepsilon$-free independence}},
    author = {Ian Charlesworth and Benoît Collins},
    journal = {Archiv der Mathematik},
    year = 2021,
    volume = 116,
    pages = {585--600}
}

@article{lambdaFreeness,
    title = {{$\Lambda$-free probability}},
    author = {M{\l}otkowski, Wojciech},
    journal = {Infinite Dimensional Analysis, Quantum Probability and Related Topics},
    volume = 7,
    number = 1,
    year = 2004,
    pages = {27--41}
}

@inproceedings{tensorPCADefinition,
    title = {{A statistical model for tensor PCA}},
    author = {Andrea Montanari and Emile Richard},
    booktitle = {Advances in Neural Information Processing Systems (NIPS)},
    year = 2014,
    volume = 27
}

@article{frischBourret,
    title = {{Parastochastics}},
    author = {Uriel Frisch and Richard Bourret},
    journal = {Journal of Mathematical Physics},
    volume = 11,
    year = 1970,
    pages = {364--390}
}

@article{sniadyModel,
    title = {Gaussian Random Matrix Models for $q$-deformed {G}aussian Variables},
    author = {{\'S}niady, Piotr},
    year = 2001,
    journal = {Communications in Mathematical Physics},
    volume = 216,
    pages = {515--537}
}

@article{strongQGaussian,
    title = {{Strongly convergent matrix models for $q$-Gaussian algebras}},
    author = {Martijn Caspers and Enli Chen},
    year = 2026,
    journal = {arXiv:2607.15100}
}

@article{mageeThomas,
    title = {{Strongly convergent unitary representations of right-angled Artin groups}},
    author = {Michael Magee and Joe Thomas},
    journal = {Duke Mathematical Journal},
    year = 2026,
    volume = 175,
    issue = 7,
    pages = {1305--1343}
}

@article{chungGramWilson,
    title = {{Quasi-random graphs}},
    author = {Fan R. K. Chung and Ronald L. Graham and Richard M. Wilson},
    journal = {Combinatorica},
    year = 1989,
    volume = 9,
    pages = {345--362}
}

@book{godsilCombinatorics,
    title = {Algebraic Graph Theory},
    author = {Chris Godsil and Gordon Royle},
    publisher = {Springer},
    year = 2001
}

@article{brownianBozejko,
    title = {An Example of a Generalized {B}rownian Motion},
    author = {Marek Bożejko and Roland Speicher},
    journal = {Communications in Mathematical Physics},
    volume = 137,
    number = 3,
    pages = {519--531},
    year = 1991
}

@incollection{qGaussianSupport,
  author    = {Bożejko, Marek and Speicher, Roland},
  title     = {An Example of a Generalized {B}rownian Motion {II}},
  booktitle = {Quantum Probability and Related Topics},
  volume    = {7},
  pages     = {67--77},
  publisher = {World Scientific},
  year      = {1992}
}

@incollection{pisierXu,
    author = {Gilles Pisier and Quanhua Xu},
    title = {Non-commutative {$L^p$-spaces}},
    booktitle = {Handbook of the Geometry of Banach Spaces},
    volume = 2,
    publisher = {North-Holland},
    pages = {1459--1517},
    year = 2003
}

@article{tensorPCAQuantum,
    title = {Classical and Quantum Algorithms for Tensor Principal Component Analysis},
    author = {Matthew B. Hastings},
    journal = {Quantum},
    volume = 4,
    pages = {237},
    year = 2020
}

@article{speicherCLT,
    title = {{A non-commutative central limit theorem}},
    author = {Roland Speicher},
    journal = {Mathematische Zeitschrift},
    volume = 209,
    pages = {55--66},
    year = 1992
}

@article{freeMonotoneTransport,
    title = {{Free monotone transport}},
    author = {Alice Guionnet and Dimitri Shlyakhtenko},
    journal = {Inventiones Mathematicae},
    year = 2014,
    volume = 197,
    pages = {613--661}
}

@article{qGaussianReloaded,
    title = {$q$-{G}aussian Processes: Non-commutative and Classical Aspects},
    author = {Marek Bożejko and Burkhard Kümmerer and Roland Speicher},
    year = 1997,
    volume = 185,
    pages = {129--154},
    journal = {Communications in Mathematical Physics}
}

@article{schmidhuberHastings2026,
    title = {A Spectral Proof of the Hypergraph {M}oore Bound},
    author = {Alexander Schmidhuber and Matthew B. Hastings},
    journal = {arXiv:2607.26028},
    year = 2026
}

\end{document}